\documentclass[11pt]{article}
\usepackage{stmaryrd}
\usepackage{amssymb}
\usepackage{amsmath}
\usepackage{amsthm}
\allowdisplaybreaks
\usepackage{amscd}
\usepackage{graphicx}
\usepackage{hyperref}
\usepackage{geometry}
\usepackage[all]{xy}
\usepackage{enumerate}
\usepackage{bookmark}
\usepackage{amsfonts}
\usepackage{fancyhdr}
\usepackage[numbers,sort&compress]{natbib}
\usepackage{epsfig}
\usepackage{picins}
\usepackage{picinpar}
\usepackage{subfigure}
\usepackage{pstricks}
\usepackage{fancyvrb}
\usepackage{dsfont}
\usepackage{mathrsfs}
\usepackage{mathpazo}
\usepackage{avant}
\usepackage{pifont}
\usepackage{longtable}

\def\vbar{\mathchoice{\vrule height6.3ptdepth-.5ptwidth.8pt\kern- .8pt}
{\vrule height6.3ptdepth-.5ptwidth.8pt\kern-.8pt} {\vrule
height4.1ptdepth-.35ptwidth.6pt\kern-.6pt} {\vrule
height3.1ptdepth-.25ptwidth.5pt\kern-.5pt}}

\newtheorem*{theoremA}{Main Theorem}
\newtheorem*{remarkA}{Remark}
\newtheorem{theorem}{Theorem}[section]
\newtheorem{definition}[theorem]{Definition}
\newtheorem{prop}[theorem]{Proposition}

\numberwithin{equation}{section}

\begin{document}

\title{Classification of five-dimensional symmetric Leibniz algebras}
\author{
Iroda Choriyeva$^{1, 2}$, \ Abror Khudoyberdiyev$^{1, 2}$
\footnote{Corresponding author, E-mail: khabror@mail.ru}
\\
\\ {\small $^1$Institute of Mathematics Uzbekistan Academy of Sciences.}
\\ {\small $^2$National University of Uzbekistan.}
}
\date{}

\maketitle

\begin{abstract}
In this paper, we give the complete classification of $5$-dimensional complex solvable symmetric Leibniz algebras.
\end{abstract}



\section*{Introduction}

A left (resp. right)
Leibniz algebra is a nonassociative algebra where the left (resp. right) multiplications are
derivations. Note that left (right) Leibniz algebras were introduced by Bloh in 1965
under the name of D-algebras \cite{Blokh}. In 1993, the Leibniz algebras were
 rediscovered by Jean-Louis Loday \cite{Loday} as a generalization of Lie algebras with no symmetry requirements. If a nonassociative algebra is both a left and right Leibniz algebra, then it is called a symmetric Leibniz algebra \cite{Mason}. These algebras had been considered in \cite{Benayadi3}, appearing in the study of some bi-invariant connections on Lie groups. In recent years, the theory of Leibniz algebras has been intensively studied.
During the last 30 years, the theory of Leibniz
algebras has been actively examined, and many results on Lie
algebras have been extended to Leibniz algebras (see, for example,
\cite{AyupovBook, Fialowski, Barnes, Gomez-Vidal, Towers, Stitzinger, LiMo}).

Recently, S. Benayadi and S. Hidri in \cite{Benayadi3} investigated the structure of left (resp. right) Leibniz algebras endowed with invariant, non-degenerate and symmetric bilinear forms,
which are called quadratic left (resp. right) Leibniz algebras. In particular, they proved that a
quadratic left (or right) Leibniz algebra is a symmetric Leibniz algebra.
At the same time, the variety of symmetric Leibniz algebras plays an important role in one-sided Leibniz algebras (more, about symmetric Leibniz algebras, see, \cite{Barreiro, Benayadi1, Benayadi3, Jibladze1, Jibladze2, Remm}).
Symmetric Leibniz algebras are related to Lie racks \cite{HamFat}, and every symmetric Leibniz algebra is flexible, power-associative and a nilalgebra with nilindex $3$ \cite{Feldvoss}.
 A symmetric Leibniz algebra under commutator and anticommutator multiplications gives a Poisson algebra \cite{Albuquerque}.

The classification, up to isomorphism, of any class of algebras is
a fundamental and very difficult problem. It is one of the first
problems that one encounters when trying to understand the
structure of a member of this class of algebras. There are many results related to the algebraic classification of small-dimensional algebras in the varieties of Lie, Leibniz, Jordan, Zinbiel and many other algebras. In particular, $5$-dimensional nilpotent, restricted Lie algebras \cite{Darijani}, $6$-dimensional nilpotent Lie algebras \cite{CicaloGraaf,Graaf},
$6$-dimensional solvable Lie algebras \cite{Turkowski},
$4$-dimensional solvable Leibniz algebras \cite{CanKhud},
$5$-dimensional solvable Leibniz algebras with three-dimensional nilradicals \cite{Khudoy-5dim},
 and some others have been described. The list of all real and complex Lie algebras up to dimension six can be found in L. Snobl and P. Winternitz's monograph \cite{Snobl}.

The purpose of the present work is to continue the study of symmetric Leibniz algebras. Since the algebraic classification of all $5$-dimensional nilpotent symmetric Leibniz algebras is given in \cite{Alvarez},  we reduce our attention to the classification of $5$-dimensional solvable symmetric Leibniz algebras. In \cite{Barreiro}, a useful characterization of
symmetric Leibniz algebras is given.
Using this characterisation, a natural method for
the classification of symmetric Leibniz algebras was given in \cite{HamFat}. It should
be noted that the center of the underlying Lie algebra plays an important
role in this method.

Using this method, we give the description of five-dimensional solvable symmetric Leibniz
algebras. For this purpose, we consider all five-dimensional solvable Lie algebras with non-zero center (even split algebras) and obtain the complete list of five-dimensional solvable symmetric Leibniz algebras.
Our main result related to the algebraic classification of the variety of solvable symmetric Leibniz
algebras is summarized below.

\begin{theoremA}
Up to isomorphism, there are infinitely many isomorphism classes of
complex $5$-dimensional solvable (non-split, non-nilpotent, non-Lie) symmetric Leibniz algebras,
described explicitly in Section 2 (see Theorems \ref{thm2.1}, \ref{thm2.2}, \ref{thm2.3}, \ref{thm2.4}, \ref{thm2.5} and \ref{thm2.6}) in terms of
$30$ one-parameter families, $8$ two-parameter families, $1$ three-parameter family and
$38$ additional isomorphism classes.
\end{theoremA}

\begin{remarkA} Since there are no  non-solvable, non-split, non-Lie $5$-dimensional symmetric Leibniz algebras, previous Main Theorem gives us the complete classification of $5$-dimensional complex symmetric Leibniz algebras.
\end{remarkA}

Throughout the paper, all the algebras (vector spaces) considered are finite-dimensional and over the field of complex numbers. Also, in tables of multiplications of algebras, we give nontrivial products only.

\section{Preliminaries}
In this section, we give the necessary definitions and preliminary results.

\begin{definition}  An algebra $(\mathcal{L},[-,-])$ over a field $F$ is
called a Lie algebra if for any $x, y, z \in \mathcal{L},$ the following
identities:
$$[x,y] = - [y,x], \quad [x,[y,z]]+[y,[z,x]]+[z,[x,y]]=0$$
hold.
\end{definition}

\begin{definition} An algebra $(\mathcal{L}, \cdot)$ is said to be a symmetric Leibniz algebra, if for any $x, y, z\in \mathcal{L}$ it satisfies the following identities:
 $$x \cdot (y\cdot z)=(x \cdot y) \cdot z+y \cdot (x \cdot z), \quad (x \cdot y) \cdot z=x\cdot (y \cdot z) + (x \cdot z) \cdot y.$$
\end{definition}

Any Lie algebra is a symmetric Leibniz algebra. However, the class of symmetric Leibniz algebras is far bigger than the class of Lie algebras.

Let $\mathcal{L}$ be a vector space equipped with a bilinear map $\cdot : \mathcal{L}\times \mathcal{L} \rightarrow \mathcal{L}$. For all $x, y \in \mathcal{L}$, we define [-,-] and $\circ$ as follows:
$$[x,y]=\frac{1} 2(x \cdot y-y\cdot x), \qquad x \circ y=\frac{1} 2 (x\cdot y+y\cdot x).$$


\begin{prop}\cite{Barreiro} Let $(\mathcal{L}, \cdot)$ be an algebra. The following assertions are equivalent:
\begin{itemize}
\item[1.] $(\mathcal{L}, \cdot)$ is a symmetric Leibniz algebra.

\item[2.] The following conditions hold:
\begin{itemize}
  \item[(a)] $(\mathcal{L}, [-,-])$ is a Lie algebra.
  \item[(b)] For any $u,v\in \mathcal{L},$  $u \circ v$ belongs to the center of $(\mathcal{L}, [-,-]).$
  \item[(c)] For any $u,v\in \mathcal{L},$ $([u,v]) \circ w=0$ and $(u\circ  v) \circ w=0.$
\end{itemize}
\end{itemize}
\end{prop}

According to this proposition, any symmetric Leibniz algebra is given by a Lie algebra $(\mathcal{L}, [-,-])$ and a bilinear symmetric form $\omega: \mathcal{L}\times \mathcal{L}\rightarrow Z(\mathcal{L})$
where $Z(\mathcal{L})$ is the center of the Lie algebra, such that for any
$x, y, z\in \mathcal{L}$
\begin{equation}\label{eq1.1} \omega([x, y], z)=\omega(\omega(x, y), z)=0.\end{equation}

Then the product of the symmetric Leibniz algebra is given by
$$u\cdot v=[u, v] + u \circ v.$$

\begin{prop}\cite{HamFat} \label{prop2} Let $(\mathcal{G} , [-,-])$ be a Lie algebra and $\omega$ and $\mu$
two solutions of \eqref{eq1.1}.
Then $(\mathcal{G}, \cdot_{\omega})$ is isomorphic to  $(\mathcal{G},\cdot_{\mu})$ if and only if there exists an automorphism $A$ of $(\mathcal{G}, [-,-])$ such that
$$\mu(u,v)=A^{(-1)}\omega(Au, Av).$$
\end{prop}

For an arbitrary symmetric Leibniz algebra $(L, \cdot)$ we define the \emph{derived} and \emph{central series} as follows:
$$L^{[1]}=L, \ L^{[s+1]}=L^{[s]} \cdot L^{[s]}, \quad s\geq 1,$$
$$L^{1}=L, \ L^{k+1}=L^{k} \cdot L, \quad k\geq 1.$$

\begin{definition} An $n$-dimensional symmetric Leibniz algebra $L$ is
called \emph{solvable (nilpotent)} if there exists $s\in {\mathbb{N}}$ ($k\in {\mathbb{N}}$) such that $L^{[s]}=0$ ($L^{k}=0$).
Such minimal number is called the \emph{index of solvability} (\emph{nilpotency}).
\end{definition}

\subsection{Classification of symmetric Leibniz algebras up to dimension four}

The classification of Leibniz algebras was obtained up to dimension four in papers \cite{4-dimnilpotent, CanKhud, Casas3-dim, Ismailov}. From the list of this classification, we can obtain the list of symmetric Leibniz algebras in low dimensions.

First, we give the list of two and three-dimensional non-Lie non-split symmetric Leibniz algebras.

\begin{longtable}{ll llllll}
\hline
{$\lambda_2$} &$:$ &  $e_1 \cdot e_1 = e_2.$ \\

\hline
${\mathcal N}_1$ &$:$ &  $e_1 \cdot e_2=  e_3,$ & $e_2\cdot e_1=e_3.$ &&& \\
${\mathcal N}_2^{\alpha}$ &$:$ & $e_1 \cdot e_1= \alpha e_3,$ & $e_2 \cdot e_1=  e_3,$ & $e_2\cdot e_2=e_3.$ &&& \\
{${\mathcal R}_1$} &$:$ &  $e_1 \cdot e_2=  e_1,$ & $e_2\cdot e_1= - e_1,$ & $e_2 \cdot e_2=  e_3.$\\
\hline
\end{longtable}

Note that, the algebras $\lambda_2,$ ${\mathcal N}_1$ and ${\mathcal N}_2^{\alpha}$ are nilpotent and ${\mathcal R}_1$ is a non-nilpotent solvable algebra.

In the following table,  we give the list of $4$-dimensional nilpotent (non 2-step nilpotent) non-Lie symmetric Leibniz algebras.
\begin{longtable}{ll llllll}
\hline
{$S_1$}&$:$ &  $e_1 \cdot e_1 = e_4,$ & $e_1\cdot e_2=e_3,$ & $e_2\cdot e_1=-e_3,$ & $e_2\cdot e_2=e_4,$ \\ & & $e_2\cdot e_3=e_4,$ & $e_3\cdot e_2=-e_4.$ \\

{$S_2$}&$:$ &  $e_1 \cdot e_1 = e_4,$ & $e_1\cdot e_2=e_3,$ & $e_2\cdot e_1=-e_3$ &  $e_2\cdot e_3=e_4,$ & $e_3\cdot e_2=-e_4.$ \\

{$S_3$}&$:$ &  $e_1 \cdot e_2 = e_3+e_4,$ & $e_2\cdot e_1=-e_3,$ &  $e_2\cdot e_3=e_4,$ & $e_3\cdot e_2=-e_4.$ \\

{$S_4$}&$:$ &  $e_1 \cdot e_2 = e_3,$ & $e_2\cdot e_1=-e_3,$ & $e_2\cdot e_2=e_4,$ &  $e_2\cdot e_3=e_4,$ & $e_3\cdot e_2=-e_4.$ \\
\hline
\end{longtable}

Here we give the list of $4$-dimensional non-Lie non-split solvable (non-nilpotent) symmetric Leibniz algebras.

\begin{longtable}{ll llllll}
\hline
{$L_1$}&$:$ &  $e_1 \cdot e_1 = e_4,$ & $e_1\cdot e_2=-e_2,$ & $e_2\cdot e_1=e_2,$ & \\
& & $e_1\cdot e_3=e_3,$ & $e_3\cdot e_1=-e_3,$ & $e_2\cdot e_3=e_4,$ & $e_3\cdot e_2=-e_4.$\\
{$L_2^{\alpha}$}&$:$ &  $e_1 \cdot e_1 = \alpha e_4,$ & $e_1\cdot e_2=e_4,$ & $e_2\cdot e_2=e_4,$ &  $e_1\cdot e_3=-e_3,$ & $e_3\cdot e_1=e_3.$ \\
{$L_3$}&$:$ &  $e_1 \cdot e_1 = 2 e_4,$ & $e_2\cdot e_2=e_4,$ &  $e_1\cdot e_3=-e_3,$ & $e_3\cdot e_1=e_3.$ \\
{$L_4^{\alpha}$}&$:$ &  $e_1 \cdot e_1 = e_4,$ & $e_1\cdot e_2=-e_2,$ & $e_2\cdot e_1=e_2,$ &  $e_1\cdot e_3=-\alpha e_3,$ & $e_3\cdot e_1=\alpha e_3.$ \\
$L_5^{\alpha}$ &$:$ &  $e_1\cdot e_2=-e_2,$ & $e_2\cdot e_1=e_2,$ &  $e_1\cdot e_3= (\alpha-1) e_4,$ &  $e_3\cdot e_1= (\alpha+1) e_4.$\\
{$L_6$}&$:$ & $e_1 \cdot e_1 = e_4,$ & $e_2\cdot e_1=e_2+e_3,$ & $e_1\cdot e_2=-e_2-e_3,$ &   $e_1\cdot e_3=-e_3,$ & $e_3\cdot e_1=e_3.$ \\
$L_7$ &$:$ & $e_1 \cdot e_1 = e_4,$ &  $e_1\cdot e_2=-e_2,$ & $e_2\cdot e_1=e_2,$ &  $e_1\cdot e_3= - e_4,$ &  $e_3\cdot e_1= e_4.$\\
\hline
\end{longtable}

It should be noted that in \cite{Alvarez} (Theorem B) the list of $4$-dimensional non-nilpotent, non-Lie symmetric Leibniz algebras is given. The algebras $\mathcal{L}_{35}^{\alpha\neq-1}$ and $\mathcal{L}_{27}$ in \cite{Alvarez} are written as $L_5^{\alpha}$ in our list and since the algebra $\mathcal{L}_{14}$ is split we omit it here.

\section{Classification of five-dimensional solvable symmetric Leibniz algebras}

Now we give the classification of five-dimensional solvable symmetric Leibniz algebras.
Since the list of $5$-dimensional nilpotent symmetric Leibniz algebras is given in \cite{Alvarez},  we reduce our attention to the classification of $5$-dimensional non-nilpotent solvable symmetric Leibniz algebras.
For this purpose we use the fact that any symmetric Leibniz algebra can be constructed by Lie algebra with non-zero center.

In this section, we use the term just solvable symmetric Leibniz algebra, instead of non-nilpotent, non-split, non-Lie solvable symmetric Leibniz algebra.


\subsection{Five-dimensional solvable symmetric Leibniz algebras, whose underlying Lie algebra is
non-split}

First, we give the description of all five-dimensional solvable symmetric Leibniz algebras, whose underlying Lie algebra is non-split. Below, we give the list of non-split $5$-dimensional complex solvable Lie algebras with non zero center \cite{Snobl}:

\begin{longtable}{ll llllll}
$\mathcal{S}_{5.1}$: & $[e_2,e_5]= e_1,$ & $[e_3,e_5]= e_2,$ & $[e_4,e_5]=e_4.$\\
  $\mathcal{S}_{5.2}$: & $[e_2,e_5]= e_1,$ & $[e_3,e_5]= e_3,$ & $[e_4,e_5]=e_3+e_4.$\\
  $\mathcal{S}_{5.3}$: & $[e_2,e_5]= e_1,$ & $[e_3,e_5]= e_3,$ & $[e_4,e_5]=\lambda e_4.$\\
  $\mathcal{S}_{5.14}$: & $[e_3,e_2]= e_1,$ & $[e_3,e_5]= e_2,$ & $[e_4,e_5]=e_4.$\\
 $\mathcal{S}_{5.15}$: & $[e_3,e_2]= e_1,$ & $[e_2,e_5]= e_2,$ & $[e_3,e_5]= -e_3,$ & $[e_4,e_5]=e_1.$\\
 $\mathcal{S}_{5.17}$:&  $[e_3,e_2]= e_1,$&  $[e_2,e_5]= e_2,$& $[e_3,e_5]= -e_3,$& $[e_4,e_5]=\lambda e_4.$\\
 $\mathcal{S}_{5.18}$: & $[e_3,e_2]= e_1,$ & $[e_2,e_5]= -e_2,$ & $[e_3,e_5]= e_3+e_4,$ & $[e_4,e_5]=e_4.$\\
 $\mathcal{S}_{5.20}$: & $[e_3,e_2]= e_1,$ & $[e_1,e_5]= e_1,$ & $[e_2,e_5]= e_2,$ & $[e_3,e_5]=e_4.$\\
 $\mathcal{S}_{5.33}$: & $[e_4,e_2]= e_1,$& $[e_4,e_3]= e_2,$ & $[e_2,e_5]=-e_2,$ & $[e_3,e_5]=-2e_3,$ & $[e_4,e_5]=e_4.$\\
 $\mathcal{S}_{5.39}$: & $[e_2,e_4]= e_2,$ & $[e_3,e_5]= e_3,$ & $[e_4,e_5]=e_1.$
\end{longtable}

\begin{theorem} \label{thm2.1} Let $\mathcal{L}$ be a five-dimensional complex solvable symmetric Leibniz algebra, whose underlying Lie algebra is non-split, then it is isomorphic to one of the following pairwise non-isomorphic algebras:
\end{theorem}
\begin{longtable}{ll lllllllllllllll}
\hline
$\mathcal{L}_{1}^{\alpha}$ & : & $e_2\cdot e_5=e_1,$ &  $e_3\cdot e_5=e_2,$ & $e_4 \cdot e_5=e_4,$ & $e_3 \cdot e_3= e_1,$ & \\ & & $e_5\cdot e_2=-e_1,$ & $e_5\cdot e_3= - e_2,$  &
  $e_5\cdot e_4=-e_4,$ & $e_5 \cdot e_5=\alpha e_1.$  \\ \hline
  $\mathcal{L}_2^{\alpha\neq0}$ & : & $e_2\cdot e_5=e_1,$ &  $e_3\cdot e_5=e_2+\alpha e_1,$ & $e_4 \cdot e_5=e_4,$ &  \\
 & & $e_5\cdot e_2=-e_1,$ & $e_5\cdot e_3= - e_2+\alpha e_1,$  &
  $e_5\cdot e_4=-e_4.$ &   \\ \hline
    $\mathcal{L}_{3}$ & : & $e_2\cdot e_5=e_1,$ &  $e_3\cdot e_5=e_2,$ & $e_4 \cdot e_5=e_4,$ & \\
 & & $e_5\cdot e_2=-e_1,$ & $e_5\cdot e_3= - e_2,$  &
  $e_5\cdot e_4=-e_4,$ &  $e_5 \cdot e_5= e_1.$ \\ \hline
$\mathcal{L}_4^{\alpha}$ &: & $e_3\cdot e_5=e_3,$ & $e_2\cdot e_5=e_1,$ &  $e_4\cdot e_5=e_3+e_4,$ & $e_2\cdot e_2=e_1,$ \\
& & $e_5 \cdot e_3=-e_3,$ & $e_5\cdot e_2=-e_1,$ &  $e_5 \cdot e_4= - e_3 - e_4,$  & $e_5\cdot e_5=\alpha e_1.$ \\
\hline
$\mathcal{L}_5^{\alpha\neq0}$ & : & $e_3 \cdot e_5=e_3,$ &  $e_2\cdot e_5=(\alpha+1) e_1,$  & $e_4\cdot e_5=e_3+e_4,$\\ & & $e_5\cdot e_3=-e_3,$ & $e_5 \cdot e_2=(\alpha-1) e_1,$ & $e_5\cdot e_4=-e_3-e_4.$\\
\hline
$\mathcal{L}_6$ &: & $e_3\cdot e_5=e_3,$ & $e_2\cdot e_5=e_1,$ &  $e_4\cdot e_5=e_3+e_4,$ & \\
& & $e_5 \cdot e_3=-e_3,$ & $e_5\cdot e_2=-e_1,$ &  $e_5 \cdot e_4= - e_3 - e_4,$  & $e_5\cdot e_5=e_1.$ \\
\hline
$\mathcal{L}_7^{\alpha\neq0,\beta}$ & :& $e_3\cdot e_5=e_3,$ & $e_2\cdot e_5=e_1,$ & $e_4 \cdot e_5= \alpha e_4,$ & $e_2 \cdot e_2=e_1,$ \\ & &
$e_5 \cdot e_3=-e_3,$ & $e_5\cdot e_2=-e_1,$ & $e_5 \cdot e_4 = - \alpha e_4,$   & $e_5 \cdot e_5= \beta e_1.$\\ \hline
$\mathcal{L}_8^{\alpha\neq0, \beta\neq0}$ & :& $e_3\cdot e_5=e_3,$ & $e_2\cdot e_5=(\beta+1)e_1,$ & $e_4 \cdot e_5= \alpha e_4,$ \\ & &
$e_5 \cdot e_3=-e_3,$ & $e_5\cdot e_2=(\beta-1)e_1,$ & $e_5 \cdot e_4 = - \alpha e_4.$   \\ \hline
$\mathcal{L}_9^{\alpha\neq0}$ & :& $e_3\cdot e_5=e_3,$ & $e_2\cdot e_5=e_1,$ & $e_4 \cdot e_5= \alpha e_4,$ \\ & &
$e_5 \cdot e_3=-e_3,$ & $e_5\cdot e_2=-e_1,$ & $e_5 \cdot e_4 = - \alpha e_4,$ & $e_5 \cdot e_5= e_1.$  \\ \hline
$\mathcal{L}_{10}^{\alpha\neq0}$ & : & $e_3\cdot e_2=e_1,$ & $e_3\cdot e_5=e_2,$ & $e_4\cdot e_5=e_4,$ & $e_3\cdot e_3= \alpha e_1,$  \\
& & $e_2 \cdot e_3=- e_1,$ & $e_5\cdot e_3=-e_2,$ & $e_5\cdot e_4=-e_4.$ \\ \hline
$\mathcal{L}_{11}^{\alpha}$ & : & $e_3\cdot e_2=e_1,$ & $e_3\cdot e_5=e_2,$ & $e_4\cdot e_5=e_4,$ & $e_3\cdot e_3= \alpha e_1,$  \\
& & $e_2 \cdot e_3=- e_1,$ & $e_5\cdot e_3=-e_2,$ & $e_5\cdot e_4=-e_4,$ & $e_5\cdot e_5= e_1.$\\ \hline
$\mathcal{L}_{12}^{\alpha, \beta}$ & : & $e_3\cdot e_2=e_1,$ &  $e_4\cdot e_5=e_4,$ & $e_3\cdot e_3= \alpha e_1,$ & $e_3\cdot e_5=e_1+e_2,$  \\
& & $e_2 \cdot e_3=- e_1,$ &  $e_5\cdot e_4=-e_4,$ & $e_5\cdot e_5= \beta e_1,$ & $e_5\cdot e_3=e_1-e_2.$\\ \hline
$\mathcal{L}_{13}^{\alpha}$ & : & $e_3\cdot e_2=e_1,$ & $e_2\cdot e_5=e_2,$ &  $e_3\cdot e_5=-e_3,$ & $e_4\cdot e_5= e_1,$   \\
& & $e_2 \cdot e_3=- e_1,$ & $e_5\cdot e_2=-e_2,$ &  $e_5\cdot e_3=e_3,$ & $e_5\cdot e_4= - e_1,$ \\
 & & $e_4\cdot e_4=  e_1,$ & $e_5\cdot e_5= \alpha e_1.$\\ \hline
$\mathcal{L}_{14}^{\alpha\neq 0 }$ & : & $e_3\cdot e_2=e_1,$ & $e_2\cdot e_5=e_2,$  & $e_4\cdot e_5= (\alpha+1)e_1,$ &  $e_3\cdot e_5=-e_3,$  \\
& & $e_2 \cdot e_3=- e_1,$ & $e_5\cdot e_2=-e_2,$ &  $e_5\cdot e_4=(\alpha -1) e_1,$ & $e_5\cdot e_3=e_3.$ \\ \hline
$\mathcal{L}_{15}$ & : & $e_3\cdot e_2=e_1,$ & $e_2\cdot e_5=e_2,$ &  $e_3\cdot e_5=-e_3,$ & $e_4\cdot e_5= e_1,$   \\
& & $e_2 \cdot e_3=- e_1,$ & $e_5\cdot e_2=-e_2,$ &  $e_5\cdot e_3=e_3,$ & $e_5\cdot e_4= - e_1,$ \\ & & $e_5\cdot e_5=  e_1.$\\ \hline
$\mathcal{L}_{16}^{\alpha}$&:& $e_3\cdot e_2=e_1,$ & $e_2\cdot e_5=e_2,$ &  $e_3\cdot e_5=-e_3,$ & $e_4\cdot e_5=\alpha e_4,$ \\ & & $e_2\cdot e_3=-e_1,$ & $e_5\cdot e_2=-e_2,$  &  $e_5 \cdot e_3=e_3,$  & $e_5\cdot e_4=-\alpha e_4,$ \\ & & $e_5\cdot e_5=e_1.$ \\ \hline
$\mathcal{L}_{17}$&:& $e_3\cdot e_2=e_1,$ & $e_2\cdot e_5=-e_2,$ & $e_3\cdot e_5=e_3+e_4,$ & $e_4\cdot e_5= e_4,$  & \\ &  &  $e_2\cdot e_3=-e_1,$ & $e_5\cdot e_2=e_2,$ &  $e_5 \cdot e_3=-e_3-e_4,$ & $e_5\cdot e_4=-e_4,$ \\ & & $ e_5\cdot e_5=e_1.$ \\ \hline
$\mathcal{L}_{18}$ & : & $e_3\cdot e_2=e_1,$ & $e_1\cdot e_5=e_1,$ & $e_2\cdot e_5=e_2,$ &  $e_3\cdot e_5= e_4,$ & \\
& & $e_2 \cdot e_3=- e_1,$ & $e_5\cdot e_1=-e_1,$ & $e_5\cdot e_2=-e_2,$ & $e_5\cdot e_3=- e_4,$ \\ & & $e_5\cdot e_5=  e_4.$ \\ \hline
$\mathcal{L}_{19}$ &:& $e_3\cdot e_2=e_1,$ & $e_1\cdot e_5=e_1,$ & $e_2\cdot e_5=e_2,$ &  $e_3\cdot e_5=2e_4,$ &  \\
& & $e_2 \cdot e_3=- e_1,$ & $e_5\cdot e_1=-e_1,$ & $e_5\cdot e_2=-e_2.$ \\ \hline
$\mathcal{L}_{20}$ & : & $e_3\cdot e_2=e_1,$ & $e_1\cdot e_5=e_1,$ & $e_2\cdot e_5=e_2,$ &  $e_3\cdot e_5=2e_4,$ &  \\
& & $e_2 \cdot e_3=- e_1,$ & $e_5\cdot e_1=-e_1,$ & $e_5\cdot e_2=-e_2,$ & $e_5\cdot e_5= e_4.$ \\ \hline
$\mathcal{L}_{21}$ & : & $e_3\cdot e_2=e_1,$ & $e_1\cdot e_5=e_1,$ & $e_2\cdot e_5=e_2,$ &  $e_3\cdot e_5=e_4,$ \\
& & $e_2 \cdot e_3=- e_1,$ & $e_5\cdot e_1=-e_1,$ & $e_5\cdot e_2=-e_2,$ & $e_5\cdot e_3=- e_4,$ \\ & & $e_3\cdot e_3=  e_4.$  \\ \hline
$\mathcal{L}_{22}$ & : & $e_3\cdot e_2=e_1,$ & $e_1\cdot e_5=e_1,$ & $e_2\cdot e_5=e_2,$ &  $e_3\cdot e_5=e_4,$  \\
& & $e_2 \cdot e_3=- e_1,$ & $e_5\cdot e_1=-e_1,$ & $e_5\cdot e_2=-e_2,$ & $e_5\cdot e_3=- e_4,$ \\ & & $e_3\cdot e_3=  e_4,$ & $e_5\cdot e_5=  e_4.$ \\ \hline
$\mathcal{L}_{23}^{\alpha}$ & : & $e_3\cdot e_2=e_1,$ & $e_1\cdot e_5=e_1,$ & $e_2\cdot e_5=e_2,$ &  $e_3\cdot e_5=2e_4,$ \\
& & $e_2 \cdot e_3=- e_1,$ & $e_5\cdot e_1=-e_1,$ & $e_5\cdot e_2=-e_2,$ & $e_5\cdot e_5= \alpha e_4,$ \\ & &  $e_3\cdot e_3= e_4.$  \\ \hline
$\mathcal{L}_{24}$ & :& $e_4 \cdot e_2=e_1,$ & $e_2\cdot e_4=-e_1,$ & $e_4\cdot e_3=e_2,$ & $e_3\cdot e_4=-e_2,$ \\ & & $e_2\cdot e_5=e_2,$ & $e_5\cdot e_2=-e_2,$ & $e_3 \cdot e_5=-2e_3,$ & $e_5\cdot e_3=2e_3,$ \\ & &$e_4\cdot e_5=e_4,$ &$e_5\cdot e_4=-e_4,$ &$e_5\cdot e_5=e_1.$ \\  \hline
$\mathcal{L}_{25}^{\alpha, \beta, \gamma}$ & :&  $e_2\cdot e_4=e_2,$ & $e_3 \cdot e_5=e_3,$ & $e_4\cdot e_5=(\beta+1)e_1,$ &$e_4\cdot e_4=\alpha e_1,$ \\ &&
$e_4 \cdot e_2=-e_2,$ &
 $e_5\cdot e_3=-e_3,$ & $e_5\cdot e_4=(\beta-1)e_1,$ &$e_5\cdot e_5=\gamma e_1.$ \\  \hline
\end{longtable}

\begin{proof}
\textbf{Case 1.} Let $\mathcal{L}$ be a  five-dimensional complex solvable symmetric Leibniz algebra, whose underlying Lie algebra is
$$\begin{array}{lllllllll}\mathcal{S}_{5.1}: & [e_2,e_5]= e_1, & [e_3,e_5]= e_2, & [e_4,e_5]=e_4.\\[1mm]\end{array}$$

Since  $Z(\mathcal{S}_{5.1}) = \{e_1\},$ then by straightforward computations, we get that the corresponding symmetric bilinear form
$\omega: \mathcal{S}_{5.1}\times \mathcal{S}_{5.1}\rightarrow Z(\mathcal{S}_{5.1})$
satisfying the equation \eqref{eq1.1}  is
$$\omega(e_3, e_3)=\alpha e_1, \quad \omega(e_3, e_5)=\beta e_1, \quad \omega(e_5, e_5)=\gamma e_1, \quad (\alpha, \beta, \gamma) \neq (0, 0,0).$$

Thus, we have the following class of symmetric Leibniz algebras
\begin{longtable}{ll lllllllllllllll}
$\mathcal{L}_{\omega}$ & : & $e_2\cdot e_5=e_1,$ &  $e_3\cdot e_5=e_2+\beta e_1,$ & $e_4 \cdot e_5=e_4,$ & $e_3 \cdot e_3=\alpha e_1,$ & \\
 & & $e_5\cdot e_2=-e_1,$ & $e_5\cdot e_3= - e_2+\beta e_1,$  &
  $e_5\cdot e_4=-e_4,$ &  $e_5 \cdot e_5=\gamma  e_1.$
\end{longtable}

By Proposition \ref{prop2}, we have that two symmetric Leibniz algebras
  $\mathcal{L}_{\omega}$ and $\mathcal{L}_{\mu}$ of this class are isomorphic if and only if there exists an automorphism $T$ of the Lie algebra $\mathcal{S}_{5.1},$ such that
$$\mu(u,v)=T^{-1}\omega(Tu,Tv).$$

 Since the matrix form of the group of automorphisms of the algebra $\mathcal{S}_{5.1}$ is
$$T= \left(
\begin{array}{clcrc}
a_{1}&a_{2}&a_{3}&0&a_{4}\\
0&a_{1}&a_{2}&0&a_{5}\\
0&0&a_{1}&0&a_{6}\\
0&0&0&a_{7}&a_{8}\\
0&0&0&0&1
\end{array}\right),$$
we have the restriction
$$\mu(e_3,e_3)=\alpha a_{1}e_1, \quad \mu(e_3,e_5)=(\alpha a_{6}+\beta)e_1, \quad \mu(e_5,e_5)=\frac{\alpha a_{6}^2+2\beta a_{6}+\gamma}{a_1}e_1.$$

Now, we consider following subcases:
\begin{itemize}
\item Let $\alpha \neq 0,$ then choosing $a_{1}=\frac{1}{\alpha},$
$a_{6}=-\frac{\beta}{\alpha},$ we get that
$$\mu(e_3,e_3)=e_1,\quad \mu(e_3,e_5)= 0,\quad \mu(e_5,e_5)= (\alpha \gamma-\beta^2)e_1$$ and obtain the algebra $\mathcal{L}_1^\alpha.$

\item Let $\alpha = 0,$ then we consider following subcases:
\begin{itemize}
	\item If $\beta \neq 0,$ then choosing $a_{6}=-\frac{\gamma}{2\beta},$ we get that
	$\mu(e_3,e_3)=0,$ $\mu(e_3,e_5)= \beta e_1,$ $\mu(e_5,e_5)= 0$ and obtain the algebra
	$\mathcal{L}_2^{\alpha\neq0}.$
	
	\item If $\beta = 0,$ then $\gamma \neq 0$ and choosing $a_{1}=\gamma,$ we get that
	$\mu(e_3,e_3)=0,$ $\mu(e_3,e_5)= 0,$ $\mu(e_5,e_5)= e_1$ and obtain the algebra $\mathcal{L}_3.$
	
\end{itemize}
\end{itemize}

\textbf{Case 2.} Let $\mathcal{L}$ be a five-dimensional complex solvable symmetric Leibniz algebra, whose underlying Lie algebra is
$$\begin{array}{lllllllll}\mathcal{S}_{5.2}: & [e_2,e_5]= e_1, & [e_3,e_5]= e_3, & [e_4,e_5]=e_3+e_4.\\[1mm]\end{array}$$

Since  $Z(\mathcal{S}_{5.2}) = \{e_1\},$ then by straightforward computations, we get that the corresponding symmetric bilinear form
$\omega: \mathcal{S}_{5.2}\times \mathcal{S}_{5.2}\rightarrow Z(\mathcal{S}_{5.2})$
satisfying the equation \eqref{eq1.1} is
$$\omega(e_2,e_2)=\alpha e_1, \quad \omega(e_2,e_5)=\beta e_1, \quad \omega(e_5,e_5)=\gamma e_1,\quad (\alpha, \beta, \gamma) \neq (0, 0,0).$$

Thus, we have the following class of symmetric Leibniz algebras
\begin{longtable}{ll lllllllllllllll}
$\mathcal{L}_{\omega}$ & : & $e_2\cdot e_5=(\beta+1)e_1,$ & $e_5\cdot e_2=(\beta-1)e_1,$ & $e_3\cdot e_5=e_3,$ & $e_5\cdot e_3= - e_3,$  \\  & & $e_4 \cdot e_5=e_3+e_4,$ & $e_5\cdot e_4=-e_3-e_4,$ & $e_2 \cdot e_2=\alpha e_1,$ & $e_5 \cdot e_5=\gamma  e_1.$
\end{longtable}

 Since the matrix form of the group of automorphisms of the algebra $\mathcal{S}_{5.2}$ is
$$T= \left(
\begin{array}{clcrc}
	a_{1}&a_{2}&0&0&a_{3}\\
	0&a_{1}&0&0&a_{4}\\
	0&0&a_{5}&a_{6}&a_{7}\\
	0&0&0&a_{5}&a_{8}\\
	0&0&0&0&1
\end{array}\right),$$
we have the restriction
$$\mu(e_2,e_2)=\alpha a_{1}e_1,\quad \mu(e_2,e_5)=(\alpha a_{4}+\beta)e_1, \quad \mu(e_5,e_5)=\frac{\alpha a_{4}^2+2\beta a_{4}+\gamma}{a_{1}}e_1.$$

Now, we consider following cases:
\begin{itemize}
\item Let $\alpha \neq 0,$ then choosing $a_{1}=\frac{1}{\alpha},$
$a_{4}=-\frac{\beta}{\alpha},$ we get that
$\mu(e_2,e_2)=e_1,$ $\mu(e_2,e_5)= 0,$ $\mu(e_5,e_5)= (\alpha \gamma-\beta^2)e_1$ and
obtain the algebra $\mathcal{L}_4^{\alpha}.$
		
\item Let $\alpha = 0.$
\begin{itemize}
	\item If $\beta \neq 0,$ then choosing $a_{4}=-\frac{\gamma}{2\beta},$ we get that
	$\mu(e_2,e_2)=0,$ $\mu(e_2,e_5)= \beta e_1,$ $\mu(e_5,e_5)= 0$ and obtain the algebra
	$\mathcal{L}_5^{\alpha\neq0}.$
		
	\item If  $\beta = 0,$ then $\gamma \neq 0$ and choosing $a_{1}=\gamma,$ we get that
	$\mu(e_2,e_2)=0,$ $\mu(e_2,e_5)= 0,$ $\mu(e_5,e_5)= e_1.$
	Thus, we have the algebra $\mathcal{L}_6.$
		
\end{itemize}
\end{itemize}

\textbf{Case 3.} Let $\mathcal{L}$ be a five-dimensional complex solvable symmetric Leibniz algebra, whose underlying Lie algebra is
$$\begin{array}{lllllllll}\mathcal{S}_{5.3}: & [e_2,e_5]= e_1, & [e_3,e_5]= e_3, & [e_4,e_5]=\lambda e_4.\\[1mm]\end{array}$$

Then $Z(\mathcal{S}_{5.3}) = \{e_1\}$ and symmetric bilinear form
$\omega: \mathcal{S}_{5.3}\times \mathcal{S}_{5.3}\rightarrow Z(\mathcal{S}_{5.3})$
satisfying the equation \eqref{eq1.1} is
$$\omega(e_2,e_2)=\alpha e_1, \quad \omega(e_2,e_5)=\beta e_1, \quad \omega(e_5,e_5)=\gamma e_1,\quad (\alpha, \beta, \gamma) \neq (0, 0,0).$$

Thus, we have the following class of symmetric Leibniz algebras
\begin{longtable}{ll lllllllllllllll}
$\mathcal{L}_{\omega}$ & : & $e_2\cdot e_5=(\beta+1)e_1,$ & $e_5\cdot e_2=(\beta-1)e_1,$ & $e_3\cdot e_5=e_3,$ & $e_5\cdot e_3= - e_3,$  \\  & & $e_4 \cdot e_5=\lambda e_4,$ & $e_5\cdot e_4=-\lambda e_4,$ & $e_2 \cdot e_2=\alpha e_1,$ & $e_5 \cdot e_5=\gamma  e_1.$
\end{longtable}

 Since the matrix form of the group of automorphisms of the algebra $\mathcal{S}_{5.3}$ is
$$T= \left(
\begin{array}{clcrc}
	a_{1}&a_{2}&0&0&a_{3}\\
	0&a_{1}&0&0&a_{4}\\
	0&0&a_{5}&0&a_{6}\\
	0&0&0&a_{5}&a_{7}\\
	0&0&0&0&1
\end{array}\right), \ \text{for} \ \lambda \neq0, \quad T= \left(
\begin{array}{clcrc}
	a_{1}&a_{2}&0&0&a_{3}\\
	0&a_{1}&0&0&a_{4}\\
	0&0&a_{5}&b_1&a_{6}\\
	0&0&b_2&a_{5}&a_{7}\\
	0&0&0&0&1
\end{array}\right) \ \text{for} \ \lambda =0,$$
we have the restriction
$$\mu(e_2,e_2)=\alpha a_{1}e_1, \quad \mu(e_2,e_5)=(\alpha a_{4}+\beta)e_1, \quad \mu(e_5,e_5)=\frac{\alpha a_{4}^2+2\beta a_{4}+\gamma}{a_{1}}e_1.$$

Now we consider following cases:
\begin{itemize}
\item Let $\alpha \neq 0,$ then choosing $a_{1}=\frac{1}{\alpha},$
$a_{4}=-\frac{\beta}{\alpha},$ we get that
$\mu(e_2,e_2)=e_1,$ $\mu(e_2,e_5)= 0,$ $\mu(e_5,e_5)= (\alpha \gamma-\beta^2)e_1$ and obtain the algebra $\mathcal{L}_7^{\alpha, \beta}.$
		
\item  Let $\alpha = 0.$
\begin{itemize}
	\item If $\beta \neq 0,$ then choosing $a_{4}=-\frac{\gamma}{2\beta},$ we get that
	$\mu(e_2,e_2)=0,$ $\mu(e_2,e_5)= \beta e_1,$ $\mu(e_5,e_5)= 0$ and obtain the algebra
	$\mathcal{L}_8^{\alpha, \beta\neq0}.$
		
	\item If $\beta = 0,$ then $\gamma \neq 0$ and choosing $a_{1}=\gamma,$ we get that
	$\mu(e_2,e_2)=0,$ $\mu(e_2,e_5)= 0,$ $\mu(e_5,e_5)= e_1$ and obtain the algebra
	$\mathcal{L}_9^{\alpha}.$
		
\end{itemize}
\end{itemize}

\textbf{Case 4.} Let $\mathcal{L}$ be a five-dimensional complex solvable symmetric Leibniz algebra, whose underlying Lie algebra is
$$\begin{array}{lllllllll}\mathcal{S}_{5.14}: & [e_3,e_2]= e_1, & [e_3,e_5]= e_2, & [e_4,e_5]=e_4.\\[1mm]\end{array}$$

Then $Z(\mathcal{S}_{5.14}) = \{e_1\}$ and
$$\omega(e_3,e_3)=\alpha e_1, \quad \omega(e_3,e_5)=\beta e_1, \quad \omega(e_5,e_5)=\gamma e_1,\quad (\alpha, \beta, \gamma) \neq (0, 0,0).$$

Thus, we have the following class of symmetric Leibniz algebras
\begin{longtable}{ll lllllllllllllll}
$\mathcal{L}_{\omega}$ & : & $e_3\cdot e_2=e_1,$ & $e_3\cdot e_5=\beta e_1+e_2,$ & $e_4\cdot e_5=e_4,$ & $e_3\cdot e_3= \alpha e_1,$  \\
& & $e_2 \cdot e_3=- e_1,$ & $e_5\cdot e_3=\beta e_1-e_2,$ & $e_5\cdot e_4=-e_4,$ & $e_5\cdot e_5= \gamma e_1.$
\end{longtable}

 Since the matrix form of the group of automorphisms of the algebra $\mathcal{S}_{5.14}$ is
$$T= \left(
\begin{array}{ccccc}
	a_{1}^2&a_{1}a_{5}&a_{2}&0&a_{3}\\
	0&a_{1}&a_{4}&0&a_{5}\\
	0&0&a_{1}&0&0\\
	0&0&0&a_{6}&a_{7}\\
	0&0&0&0&1
\end{array}\right),$$
we have the restriction
$$\mu(e_3,e_3)=\alpha e_1,\quad \mu(e_3,e_5)=\frac{\beta}{a_{1}}e_1, \quad \mu(e_5,e_5)=\frac{\gamma}{a_{1}^2}e_1.$$

Now we consider following cases:
\begin{itemize}
\item If $\beta =0, \gamma =0,$ then we have the algebra $\mathcal{L}_{10}^{\alpha}.$
	
\item If $\beta =0,$ $\gamma \neq 0,$ then have the algebra $\mathcal{L}_{11}^{\alpha}.$
		
\item If $\beta \neq 0,$ then choosing $a_1=\beta,$ we obtain the algebra $\mathcal{L}_{12}^{\alpha, \beta}.$
		
\end{itemize}

\textbf{Case 5.} Let $\mathcal{L}$ be a  five-dimensional complex solvable symmetric Leibniz algebra, whose underlying Lie algebra is
$$\begin{array}{lllllllll}\mathcal{S}_{5.15}: & [e_3,e_2]= e_1, & [e_2,e_5]= e_2, & [e_3,e_5]=-e_3, & [e_4,e_5]=e_1.\end{array}$$

Then  $Z(\mathcal{S}_{5.15}) = \{e_1\}$ and
$$\omega(e_4,e_4)=\alpha e_1, \quad \omega(e_4,e_5)=\beta e_1, \quad \omega(e_5,e_5)=\gamma e_1,$$
where $(\alpha, \beta, \gamma) \neq (0, 0,0).$

Thus, we have the following class of symmetric Leibniz algebras
\begin{longtable}{ll lllllllllllllll}
$\mathcal{L}_{\omega}$ & : & $e_3\cdot e_2=e_1,$ & $e_2\cdot e_5=e_2,$ &  $e_3\cdot e_5=-e_3,$ & $e_4\cdot e_5=(\beta+1)e_1,$ & $e_4\cdot e_4= \alpha e_1,$  \\
& & $e_2 \cdot e_3=- e_1,$ & $e_5\cdot e_2=-e_2,$ &  $e_5\cdot e_3=e_3,$ & $e_5\cdot e_4=(\beta-1)e_1,$ & $e_5\cdot e_5= \gamma e_1.$
\end{longtable}

By Proposition \ref{prop2}, we have that two symmetric Leibniz algebras
  $\mathcal{L}_{\omega}$ and $\mathcal{L}_{\mu}$ of this class are isomorphic if and only if there exists an automorphism $T$ of the Lie algebra $\mathcal{S}_{5.15},$ such that
$\mu(u,v)=T^{-1}\omega(Tu,Tv).$ Since the matrix form of the group of automorphisms of the algebra $\mathcal{S}_{5.15}$ is
$$T= \left(
\begin{array}{ccccc}
	a_{3}a_5&-a_{3}a_{6}&-a_{4}a_{5}&a_{1}&a_{2}\\
	0&a_{3}&0&0&a_{4}\\
	0&0&a_{5}&0&a_{6}\\
	0&0&0&a_{3}a_5&a_{7}\\
	0&0&0&0&1
\end{array}\right),$$
we have the restriction
$$\mu(e_4,e_4)=a_{3} a_5 \alpha e_1, \quad \mu(e_4,e_5)=(a_{7}\alpha+\beta)e_1, \quad
\mu(e_5,e_5)=\frac{a_{7}^2\alpha + 2a_{7}\beta + \gamma}{a_{3}a_5}e_1.$$

Now we consider following cases:
\begin{itemize}
\item If $\alpha \neq 0,$ then choosing $a_{3}=\frac{1}{\alpha a_5},$
$a_{7}=-\frac{\beta}{\alpha},$ we have the algebra $\mathcal{L}_{13}^{\alpha}.$
		
\item If $\alpha = 0$ and $\beta \neq 0,$ then choosing $a_{7}=-\frac{\gamma}{2\beta},$ we obtain the algebra $\mathcal{L}_{14}^{\alpha\neq0}.$
\item  If $\alpha = 0$ and $\beta = 0,$ then $\gamma \neq 0$ and choosing $a_{7}=\gamma,$ we obtain the algebra $\mathcal{L}_{15}.$
\end{itemize}

\textbf{Case 6.} Let $\mathcal{L}$ be a five-dimensional complex solvable symmetric Leibniz algebra, whose underlying Lie algebra is
$$\begin{array}{lllllllll}\mathcal{S}_{5.17}: & [e_3,e_2]= e_1, & [e_2,e_5]= e_2, & [e_3,e_5]=-e_3, & [e_4,e_5]=\lambda e_4.\\[1mm]\end{array}$$

Since  $Z(\mathcal{S}_{5.17}) = \{e_1\},$ then by straightforward computations, we get that the corresponding symmetric bilinear form
$\omega: \mathcal{S}_{5.17}\times \mathcal{S}_{5.17}\rightarrow Z(\mathcal{S}_{5.17})$
satisfying the equation \eqref{eq1.1} is
$$\omega(e_5,e_5)=\alpha e_1,$$
where $\alpha \neq 0.$ Hence, in this case, we obtain the algebra $\mathcal{L}_{16}^{\alpha\neq0}.$

\textbf{Case 7.} Let $\mathcal{L}$ be a five-dimensional complex solvable symmetric Leibniz algebra, whose underlying Lie algebra is
$$\begin{array}{lllllllll}\mathcal{S}_{5.18}: & [e_3,e_2]= e_1, & [e_2,e_5]=-e_2, & [e_3,e_5]= e_3+e_4, & [e_4,e_5]=e_4.\\[1mm]\end{array}$$

Since  $Z(\mathcal{S}_{5.18}) = \{e_1\},$ then by straightforward computations, we get that the corresponding symmetric bilinear form
$\omega: \mathcal{S}_{5.18}\times \mathcal{S}_{5.18}\rightarrow Z(\mathcal{S}_{5.18})$
satisfying the equation \eqref{eq1.1} is
$$\omega(e_5,e_5)=\alpha e_1 $$
where $\alpha \neq 0.$ In this case we obtain the algebra $\mathcal{L}_{17}^{\alpha}.$

\textbf{Case 8.} Let $\mathcal{L}$ be a  five-dimensional complex solvable symmetric Leibniz algebra, whose underlying Lie algebra is
$$\begin{array}{lllllllll}\mathcal{S}_{5.20}: & [e_3,e_2]= e_1, & [e_1,e_5]=e_1, & [e_2,e_5]=e_2, & [e_3,e_5]=e_4.\\[1mm]\end{array}$$

Since  $Z(\mathcal{S}_{5.20}) = \{e_4\},$ then by straightforward computations, we get that the corresponding symmetric bilinear form
$\omega: \mathcal{S}_{5.20}\times \mathcal{S}_{5.20}\rightarrow Z(\mathcal{S}_{5.20})$
satisfying the equation \eqref{eq1.1} is
$$\omega(e_3,e_3)=\alpha e_4, \quad \omega(e_3,e_5)=\beta e_4,\quad \omega(e_5,e_5)=\gamma e_4,$$
where $(\alpha,\beta ,\gamma) \neq (0,0,0).$

Thus, we have the following class of symmetric Leibniz algebras
\begin{longtable}{ll lllllllllllllll}
$\mathcal{L}_{\omega}$ & : & $e_3\cdot e_2=e_1,$ & $e_1\cdot e_5=e_1,$ & $e_2\cdot e_5=e_2,$ &  $e_3\cdot e_5=(\beta+1)e_4,$ &  $e_3\cdot e_3= \alpha e_4,$  \\
& & $e_2 \cdot e_3=- e_1,$ & $e_5\cdot e_1=-e_1,$ & $e_5\cdot e_2=-e_2,$ & $e_5\cdot e_3=(\beta-1) e_4,$ & $e_5\cdot e_5= \gamma e_4.$
\end{longtable}

Since the matrix form of the group of automorphisms of the algebra $\mathcal{S}_{5.20}$ is
$$T= \left(
\begin{array}{ccccc}
	a_{3}a_{5}&a_{1}&-a_{4}a_{5}&0&a_{2}\\
	0&a_{3}&0&0&a_{4}\\
	0&0&a_{5}&0&0\\
	0&0&a_{6}&a_{7}&a_{8}\\
	0&0&0&0&1
\end{array}\right),$$
we have the restriction
$$\mu(e_3,e_3)=\frac{\alpha a_{5}^2} {a_{7}}  e_4, \quad \mu(e_3,e_5)=\frac{\beta a_{5}} {a_{7}}  e_4, \quad \mu(e_5,e_5)= \frac{\gamma} {a_{7}}  e_4.$$
Now we consider following cases:

\begin{itemize}
\item Let $\alpha = \beta = 0$ and $\gamma \neq 0,$ then choosing $a_{7}=\gamma,$
 we have the algebra $\mathcal{L}_{18}.$
\item Let $\alpha =0,$ $ \beta \neq 0$ and $\gamma = 0,$ then choosing $a_{5}=\frac {a_{7}}{\beta},$ we obtain the algebra $\mathcal{L}_{19}.$
\item Let $\alpha =0,$ $ \beta \neq 0$ and $\gamma \neq 0,$ then choosing $a_{5}=\frac {\gamma}{\beta},$ $a_{7}= \gamma,$
we obtain the algebra $\mathcal{L}_{20}.$
\item Let $\alpha \neq0,$ $ \beta = 0$ and $\gamma = 0,$ then choosing $a_{7}=\alpha a_{5}^2,$
we obtain the algebra $\mathcal{L}_{21}.$
\item Let $\alpha \neq0,$ $ \beta = 0$ and $\gamma \neq 0,$ then choosing $a_{7}=\gamma,$
$a_{5}=\sqrt{\frac {\gamma}{\alpha}},$
we obtain the algebra
$\mathcal{L}_{22}.$
\item Let $\alpha \neq0,$ $ \beta \neq 0,$ then choosing $a_{5}=\frac {\beta}{\alpha},$
$a_{7}=\frac {\beta^2}{\alpha},$
we obtain the algebra $\mathcal{L}_{23}^{\alpha}.$
\end{itemize}

\textbf{Case 9.} Let $\mathcal{L}$ be a five-dimensional complex solvable symmetric Leibniz algebra, whose underline Lie algebra is $$\begin{array}{lllllllll}\mathcal{S}_{5.33}: & [e_4,e_2]= e_1, & [e_4,e_3]=e_2, & [e_2,e_5]=-e_2, & [e_3,e_5]=-2e_3, & [e_4,e_5]=e_4.\\[1mm]\end{array}$$

Since  $Z(\mathcal{S}_{5.33}) = \{e_1\},$ then by straightforward computations, we get that the corresponding symmetric bilinear form
$\omega: \mathcal{S}_{5.33}\times \mathcal{S}_{5.33}\rightarrow Z(\mathcal{S}_{5.33})$
satisfying the equation \eqref{eq1.1} is
$$\omega(e_5,e_5)=\alpha e_1,$$
where $\alpha \neq 0.$ Hence, in this case, we obtain the algebra $\mathcal{L}_{24}.$

\textbf{Case 10.} Let $\mathcal{L}$ be a  five-dimensional complex solvable symmetric Leibniz algebra, whose underlying Lie algebra is
$$\begin{array}{lllllllll}\mathcal{S}_{5.39}: & [e_2,e_4]= e_2, & [e_3,e_5]= e_3, & [e_4,e_5]=e_1.\\[1mm]\end{array}$$
Since  $Z(\mathcal{S}_{5.39}) = \{e_1\},$ then by straightforward computations, we get that the corresponding symmetric bilinear form
$\omega: \mathcal{S}_{5.39}\times \mathcal{S}_{5.39}\rightarrow Z(\mathcal{S}_{5.39})$
satisfying the equation \eqref{eq1.1} is
$$\omega(e_4,e_4)=\alpha e_1, \quad \omega(e_4,e_5)=\beta e_1, \quad \omega(e_5,e_5)=\gamma e_1,$$
where $(\alpha, \beta, \gamma) \neq (0, 0,0).$

The group of automorphisms of the algebra $\mathcal{S}_{5.39}$ is
$$T= \left(
\begin{array}{ccccc}
	1&0&0&a_{1}&a_{2}\\
	0&a_{3}&0&a_{4}&a_5\\
	0&0&a_{6}&a_{7}&a_{8}\\
	0&0&0&1&0\\
	0&0&0&0&1
\end{array}\right) \ \quad \text{or} \quad T= \left(
\begin{array}{ccccc}
	-1&0&0&a_{1}&a_{2}\\
	0&0&a_{3}&a_{4}&a_5\\
	0&a_{6}&0&a_{7}&a_{8}\\
	0&0&0&0&1\\
	0&0&0&1&0
\end{array}\right),$$
 and we have the family of algebras
$\mathcal{L}_{25}^{\alpha, \beta, \gamma}$ with three parameters $\alpha, \beta, \gamma,$ where $\mathcal{L}_{25}^{\alpha, \beta, \gamma}\simeq \mathcal{L}_{25}^{-\gamma, -\beta, -\alpha}.$

\end{proof}

\subsection{Five-dimensional solvable symmetric Leibniz algebras, whose underlying Lie algebra is $\mathcal{S}_{4} \oplus \mathbb{C}$}

In this subsection, we give the classification of five-dimensional solvable symmetric Leibniz algebras, whose underlying Lie algebra is a direct sum of four-dimensional non-split algebra and one-dimensional abelian ideal, i.e., $\mathcal{S}_{4} \oplus \mathbb{C}$.

For this purpose, we give the list of complex $4$-dimensional non-split solvable Lie algebras  \cite{Snobl}:
\begin{longtable}{llllllllllllllll}
$\mathcal{S}_{4.1}$&: & $[e_2,e_4]=e_1,$ & $[e_3,e_4]=e_3.$ \\
$\mathcal{S}_{4.2}$&: & $[e_1,e_4]=e_1,$ & $[e_2,e_4]=e_1+e_2,$ & $[e_3,e_4]=e_2+e_3.$ \\
$\mathcal{S}_{4.3}$&: & $[e_1,e_4]=e_1,$ & $[e_2,e_4]= a e_2,$ & $[e_3,e_4]= b e_3,$ & $0<|b|\leq |a| \leq 1.$ \\
$\mathcal{S}_{4.4}$&: & $[e_1,e_4]=e_1,$ & $[e_2,e_4]= e_1+e_2,$ & $[e_3,e_4]= a e_3,$ & $a\neq 0.$ \\[1mm]
$\mathcal{S}_{4.6}$&: & $[e_2,e_3]=e_1,$ & $[e_2,e_4]=e_2,$ & $[e_3,e_4]=-e_3.$\\
$\mathcal{S}_{4.8}$&: & $[e_2,e_3]=e_1,$ & $[e_1,e_4]=(1+a)e_1,$ & $[e_2,e_4]=e_2,$ & $[e_3,e_4]=a e_3,$ & $0< |a| \leq 1.$ \\
$\mathcal{S}_{4.10}$ & : & $[e_2,e_3]=e_1,$ & $[e_1,e_4]=2e_1,$ & $[e_2,e_4]=e_2,$ & $[e_3,e_4]= e_2+e_3.$\\
$\mathcal{S}_{4.11}$&: & $[e_2,e_3]=e_1,$ & $[e_1,e_4]=e_1,$ & $[e_2,e_4]=e_2.$\\
$\mathcal{S}_{4.12}$&: & $[e_1,e_3]=e_1,$ & $[e_2,e_4]=e_2.$ \\[1mm]
\end{longtable}

\begin{theorem} \label{thm2.2} Let $\mathcal{L}$ be a complex five-dimensional solvable symmetric Leibniz algebra, whose underlying Lie algebra is $\mathcal{S}_{4} \oplus \mathbb{C}$, then it is isomorphic to one of the following pairwise non-isomorphic algebras
\begin{longtable}{llllllllllllllll}
\hline
$\mathcal{L}_{26}^{\alpha,\beta}$ & $:$ & $e_3 \cdot e_4= e_3,$ & $e_2 \cdot e_4= (\alpha +1)e_1,$ & $e_2 \cdot e_2= e_5,$ &\\
 & & $e_4 \cdot e_3 = - e_3,$ & $e_4 \cdot e_2 = (\alpha  - 1) e_1,$ & $e_4 \cdot e_4 = \beta e_1 + e_5.$ \\ \hline
$\mathcal{L}_{27}^{\alpha}$ & $:$ & $e_3 \cdot e_4= e_3,$ & $e_2 \cdot e_4= (\alpha +1)e_1,$ & $e_2 \cdot e_2= e_5,$ &\\
 & & $e_4 \cdot e_3 = - e_3,$ & $e_4 \cdot e_2 = (\alpha  - 1) e_1,$ & $e_4 \cdot e_4 = e_1.$ \\ \hline
$\mathcal{L}_{28}^{\alpha}$ & $:$ & $e_3 \cdot e_4= e_3,$ & $e_2 \cdot e_4= (\alpha +1)e_1,$ & $e_2 \cdot e_2= e_5,$ &\\
 & & $e_4 \cdot e_3 = - e_3,$ & $e_4 \cdot e_2 = (\alpha  - 1) e_1.$ \\ \hline
$\mathcal{L}_{29}^{\alpha}$ & $:$ & $e_3 \cdot e_4= e_3,$ & $e_2 \cdot e_4= e_1 + e_5,$ & $e_2 \cdot e_2= e_1,$ &\\
 & & $e_4 \cdot e_3 = - e_3,$ & $e_4 \cdot e_2 = - e_1 +e_5,$ & $e_4 \cdot e_4 = \alpha e_1.$ \\ \hline
$\mathcal{L}_{30}$ & $:$ & $e_3 \cdot e_4= e_3,$ & $e_2 \cdot e_4= e_1 + e_5,$ & $e_4 \cdot e_4= e_1,$ &\\
 & & $e_4 \cdot e_3 = - e_3,$ & $e_4 \cdot e_2 = - e_1 +e_5.$ \\ \hline
 $\mathcal{L}_{31}$ & $:$ & $e_3 \cdot e_4= e_3,$ & $e_2 \cdot e_4= e_1 + e_5,$ & \\
 & & $e_4 \cdot e_3 = - e_3,$ & $e_4 \cdot e_2 = - e_1 +e_5.$ \\ \hline
 $\mathcal{L}_{32}$ & $:$ & $e_3 \cdot e_4= e_3,$ & $e_2 \cdot e_4= e_1,$ & $e_2 \cdot e_2= e_1,$ &\\
 & & $e_4 \cdot e_3 = - e_3,$ & $e_4 \cdot e_2 = - e_1,$ & $e_4 \cdot e_4 = e_5.$ \\ \hline
$\mathcal{L}_{33}^{\alpha}$ & $:$ & $e_3 \cdot e_4= e_3,$ & $e_2 \cdot e_4= (\alpha +1)e_1,$ & $e_4 \cdot e_4= e_5,$ &\\
 & & $e_4 \cdot e_3 = - e_3,$ & $e_4 \cdot e_2 = (\alpha  - 1) e_1.$ \\ \hline
$\mathcal{L}_{34}^{\alpha}$ & $:$ & $e_3 \cdot e_4= e_3,$ & $e_2 \cdot e_4= e_1,$ & $e_2 \cdot e_2= e_1,$ & $e_5 \cdot e_5= e_1,$ &\\
 & & $e_4 \cdot e_3 = - e_3,$ & $e_4 \cdot e_2 = -e_1,$ & $e_4 \cdot e_4 = \alpha e_1.$
\\ \hline
$\mathcal{L}_{35}^{\alpha}$ & $:$ & $e_3 \cdot e_4= e_3,$ & $e_2 \cdot e_4= (\alpha + 1)e_1,$ & $e_5 \cdot e_5= e_1,$ &\\
 & & $e_4 \cdot e_3 = - e_3,$ & $e_4 \cdot e_2 = (\alpha - 1)e_1.$
\\ \hline
$\mathcal{L}_{36}$ & $:$ & $e_3 \cdot e_4= e_3,$ & $e_2 \cdot e_4= e_1,$ & $e_5 \cdot e_5= e_1,$ &\\
 & & $e_4 \cdot e_3 = - e_3,$ & $e_4 \cdot e_2 = -e_1,$ & $e_4 \cdot e_4 = e_1.$
 \\ \hline
$\mathcal{L}_{37}$ & $:$ & $e_3 \cdot e_4= e_3,$ & $e_2 \cdot e_4= e_1,$ & $e_2 \cdot e_5= e_1,$ \\
 & & $e_4 \cdot e_3 = - e_3,$ & $e_4 \cdot e_2 = -e_1,$ & $e_5 \cdot e_2= e_1.$
 \\ \hline
$\mathcal{L}_{38}$ & $:$ & $e_3 \cdot e_4= e_3,$ & $e_2 \cdot e_4= e_1,$ & $e_2 \cdot e_5= e_1,$ & $e_4 \cdot e_4= e_1,$ \\
 & & $e_4 \cdot e_3 = - e_3,$ & $e_4 \cdot e_2 = -e_1,$ & $e_5 \cdot e_2= e_1.$
 \\ \hline
$\mathcal{L}_{39}$ & $:$ & $e_3 \cdot e_4= e_3,$ & $e_2 \cdot e_4= e_1,$ & $e_4 \cdot e_5= e_1,$ \\
 & & $e_4 \cdot e_3 = - e_3,$ & $e_4 \cdot e_2 = -e_1,$ & $e_5 \cdot e_4= e_1.$
\\ \hline
$\mathcal{L}_{40}$ & $:$ & $e_3 \cdot e_4= e_3,$ & $e_2 \cdot e_4= e_1,$ & $e_4 \cdot e_5= e_1,$& $e_2 \cdot e_2= e_1,$\\
 & & $e_4 \cdot e_3 = - e_3,$ & $e_4 \cdot e_2 = -e_1,$ & $e_5 \cdot e_4= e_1.$
   \\ \hline
 $\mathcal{L}_{41}$ & $:$ & $e_1 \cdot e_4= e_1,$ & $e_2 \cdot e_4= e_1 + e_2,$ & $e_3 \cdot e_4= e_2 + e_3,$ &\\
  & & $e_4 \cdot e_1 = - e_1,$ & $e_4 \cdot e_2 = - e_1 - e_2,$ & $e_4 \cdot e_3 = - e_2 - e_3,$  & $e_4 \cdot e_4=e_5.$ \\ \hline
  $\mathcal{L}_{42}^{a, b}$ & $:$ & $e_1 \cdot e_4= e_1,$ & $e_2 \cdot e_4= a e_2,$ & $e_3 \cdot e_4= b e_3,$ &\\
  & & $e_4 \cdot e_1 = - e_1,$ & $e_4 \cdot e_2 = - a e_2,$ & $e_4 \cdot e_3 = - be_3,$  & $e_4 \cdot e_4=e_5.$ \\ \hline
  $\mathcal{L}_{43}^{\alpha \neq0}$ & $:$ & $e_1 \cdot e_4= e_1,$ & $e_2 \cdot e_4= e_1+ e_2,$ & $e_3 \cdot e_4= \alpha e_3,$ &\\
  & & $e_4 \cdot e_1 = - e_1,$ & $e_4 \cdot e_2 = - e_1- e_2,$ & $e_4 \cdot e_3 = - \alpha e_3,$  & $e_4 \cdot e_4=e_5.$ \\ \hline
  $\mathcal{L}_{44}$ & $:$ & $e_2 \cdot e_3= e_1,$ & $e_2 \cdot e_4= e_2,$ & $e_3 \cdot e_4= -e_3,$ & $e_4 \cdot e_4= e_5,$ &\\
 & & $e_3 \cdot e_2 = - e_1,$ & $e_4 \cdot e_2 = - e_2,$ & $e_4 \cdot e_3= e_3.$\\ \hline
$\mathcal{L}_{45}$ & $:$ & $e_2 \cdot e_3= e_1,$ & $e_2 \cdot e_4= e_2,$ & $e_3 \cdot e_4= - e_3,$ \\
 & & $e_3 \cdot e_2 = - e_1,$ & $e_4 \cdot e_2 = - e_2,$ & $e_4 \cdot e_3 = e_3,$ & $e_5 \cdot e_5= e_1.$ \\ \hline
 $\mathcal{L}_{46}$ & $:$ & $e_2 \cdot e_3= e_1,$ & $e_2 \cdot e_4= e_2,$ & $e_3 \cdot e_4= - e_3,$ & $e_4 \cdot e_4= e_1,$ &\\
 & & $e_3 \cdot e_2 = - e_1,$ & $e_4 \cdot e_2 = - e_2,$ & $e_4 \cdot e_3 = e_3,$ & $e_5 \cdot e_5= e_1.$ \\ \hline
  $\mathcal{L}_{47}$ & $:$ & $e_2 \cdot e_3= e_1,$ & $e_2 \cdot e_4= e_2,$ & $e_3 \cdot e_4= - e_3,$ & $e_4 \cdot e_5=  e_1,$ \\
 & & $e_3 \cdot e_2 = - e_1,$ & $e_4 \cdot e_2 = - e_2,$ & $e_4 \cdot e_3 = e_3,$ & $e_5 \cdot e_4= e_1.$    \\ \hline
  $\mathcal{L}_{48}^{a}$ & $:$ & $e_2 \cdot e_3= e_1,$ & $e_1 \cdot e_4= (a+1) e_1,$ & $e_2 \cdot e_4= e_2,$ & $e_3 \cdot e_4= a e_3,$ &\\
 & & $e_3 \cdot e_2= -e_1,$ & $e_4 \cdot e_1 = -(a+1) e_1,$ & $e_4 \cdot e_2 = - e_2,$ & $e_4 \cdot e_3 = - ae_3,$ \\ & & $e_4 \cdot e_4=e_5.$ \\ \hline
 $\mathcal{L}_{49}$ & $:$ & $e_2 \cdot e_3= e_1,$ & $e_1 \cdot e_4= 2 e_1,$ & $e_2 \cdot e_4= e_2,$ & $e_3 \cdot e_4= e_2+e_3,$ &\\
 & & $e_3 \cdot e_2= -e_1,$ & $e_4 \cdot e_1 = -2 e_1,$ & $e_4 \cdot e_2 = - e_2,$ & $e_4 \cdot e_3 = - e_2-e_3,$  \\ & & $e_4 \cdot e_4=e_5.$\\ \hline
$\mathcal{L}_{50}$ & $:$ & $e_3\cdot e_2=e_1,$ & $e_1\cdot e_4=e_1,$ & $e_2\cdot e_4=e_2,$ &  \\
& & $e_2 \cdot e_3=- e_1,$ & $e_4\cdot e_1=-e_1,$ & $e_4\cdot e_2=-e_2,$ & $e_4\cdot e_4= e_5.$ \\ \hline
$\mathcal{L}_{51}$ & $:$ & $e_3\cdot e_2=e_1,$ & $e_1\cdot e_4=e_1,$ & $e_2\cdot e_4=e_2,$ &  $e_3\cdot e_4= e_5,$  \\
& & $e_2 \cdot e_3=- e_1,$ & $e_4\cdot e_1=-e_1,$ & $e_4\cdot e_2=-e_2,$ & $e_4\cdot e_3= e_5.$  \\ \hline
$\mathcal{L}_{52}$ & $:$ & $e_3\cdot e_2=e_1,$ & $e_1\cdot e_4=e_1,$ & $e_2\cdot e_4=e_2,$ &  $e_3\cdot e_4=  e_5,$  \\
& & $e_2 \cdot e_3=- e_1,$ & $e_4\cdot e_1=-e_1,$ & $e_4\cdot e_2=-e_2,$ & $e_4\cdot e_3= e_5,$ \\ & & $e_4\cdot e_4= e_5.$ \\ \hline
$\mathcal{L}_{53}$ & $:$ & $e_3\cdot e_2=e_1,$ & $e_1\cdot e_4=e_1,$ & $e_2\cdot e_4=e_2,$ &  $e_3\cdot e_3= e_5,$ &   \\
& & $e_2 \cdot e_3=- e_1,$ & $e_4\cdot e_1=-e_1,$ & $e_4\cdot e_2=-e_2.$  \\ \hline
$\mathcal{L}_{54}$ & $:$ & $e_3\cdot e_2=e_1,$ & $e_1\cdot e_4=e_1,$ & $e_2\cdot e_4=e_2,$ &  $e_3\cdot e_3= e_5,$  \\
& & $e_2 \cdot e_3=- e_1,$ & $e_4\cdot e_1=-e_1,$ & $e_4\cdot e_2=-e_2,$ & $e_4\cdot e_4= e_5.$ \\ \hline
$\mathcal{L}_{55}^{\alpha}$ & $:$ & $e_3\cdot e_2=e_1,$ & $e_1\cdot e_4=e_1,$ & $e_2\cdot e_4=e_2,$ &   $e_3\cdot e_4=  e_5,$  \\
& & $e_2 \cdot e_3=- e_1,$ & $e_4\cdot e_1=-e_1,$ & $e_4\cdot e_2=-e_2,$ & $e_4\cdot e_3= e_5,$ \\ & &  $e_3\cdot e_3= e_5,$ & $e_4\cdot e_4= \alpha e_5.$ \\ \hline
$\mathcal{L}_{56}$ & $:$ & $e_1\cdot e_3=e_1,$ & $e_2\cdot e_4=e_2,$ &  \\
& & $e_3\cdot e_1=-e_1,$ & $e_4\cdot e_2=-e_2,$ & $e_3\cdot e_3= e_5.$ \\ \hline
$\mathcal{L}_{57}^{\alpha}$ & $:$ & $e_1\cdot e_3=e_1,$ & $e_2\cdot e_4=e_2,$ &  $e_3\cdot e_3=\alpha e_5,$ &  $e_3\cdot e_4=  e_5,$  \\
& & $e_3\cdot e_1=-e_1,$ & $e_4\cdot e_2=-e_2,$ & $e_4\cdot e_3= e_5.$ \\ \hline
$\mathcal{L}_{58}^{\alpha, \beta}$ & $:$ & $e_1\cdot e_3=e_1,$ & $e_2\cdot e_4=e_2,$ &  $e_3\cdot e_3=\alpha e_5,$ &  $e_3\cdot e_4= \beta e_5,$  \\
& & $e_3\cdot e_1=-e_1,$ & $e_4\cdot e_2=-e_2,$ & $e_4\cdot e_3=\beta e_5,$ & $e_4\cdot e_4= e_5.$
\end{longtable}

\end{theorem}

\begin{proof}

\textbf{Case 1.} Let $\mathcal{L}$ be a five-dimensional complex solvable symmetric Leibniz algebra, whose underlying Lie algebra is
$$\begin{array}{lllllll}\mathcal{S}_{4.1}\oplus \mathbb{C}: & [e_2,e_4]=e_1, & [e_3,e_4]=e_3.\end{array}$$

Then, we have $Z(\mathcal{S}_{4.1}\oplus \mathbb{C}) = \{e_1,e_5\}$ and by straightforward computations we get that the corresponding symmetric bilinear form
satisfying the equation \eqref{eq1.1} is
$$\omega(e_2,e_2)=\alpha_{1}e_1+\beta_{1}e_5,\quad \omega(e_2,e_4)=\alpha_{2}e_1+\beta_{2}e_5,\quad \omega(e_4,e_4)=\alpha_{3}e_1+\beta_{3}e_5,$$
where $(\beta_{1},\beta_{2},\beta_{3})\neq (0, 0 , 0)$
or
$$\omega(e_2,e_2)=\alpha_{1}e_1,\quad \omega(e_2,e_4)=\alpha_{2}e_1,\quad \omega(e_4,e_4)=\alpha_{3}e_1,$$
$$\omega(e_2,e_5)=\alpha_{4}e_1,\quad \omega(e_4,e_5)=\alpha_{5}e_1,\quad \omega(e_5,e_5)=\alpha_{6}e_1,$$
where $(\alpha_{4}, \alpha_{5},\alpha_{6})\neq (0, 0 , 0).$

Thus, we consider two subcases. In the first subcase, we have the class of symmetric Leibniz algebras
\begin{longtable}{ll lllllllllllllll}
$\mathcal{L}_{\omega}$ & $:$ & $e_2\cdot e_4=e_1,$ & $e_3\cdot e_4=e_3,$ & $e_2\cdot e_4=\alpha_{2}e_1+\beta_{2}e_5,$ &  $e_2\cdot e_2=\alpha_{1}e_1+\beta_{1}e_5,$  \\
& & $e_4 \cdot e_2=- e_1,$ & $e_4\cdot e_3=-e_3,$ & $e_4\cdot e_2=\alpha_{2}e_1+\beta_{2}e_5,$ & $e_4\cdot e_4= \alpha_{3}e_1+\beta_{3}e_5.$
\end{longtable}

Since the matrix form of the group of automorphisms of the algebra  $\mathcal{S}_{4.1}\oplus \mathbb{C}$ is
$$T= \left(
\begin{array}{ccccc}
a_1&a_2&0&a_3&a_4\\
0&a_1&0&a_5&0\\
0&0&a_6&a_7&0\\
0&0&0&1&0\\
0&a_8&0&a_9&a_{10}
\end{array}\right),$$
for the first subcase we have the restriction
$$\begin{array}{llll}
\mu(e_2,e_2)&=&\frac{a_1(a_{10}\alpha_{1}-a_4\beta_{1})}{a_{10}}e_1+\frac{a_1^2 \beta_{1}}{a_{10}}e_5,\\[1mm]
\mu(e_2,e_4)&=&\frac{a_5(a_{10}\alpha_{1}-a_4\beta_{1})+a_{10}\alpha_{2}-a_4\beta_{2}}{a_{10}}e_1+
\frac{a_{1}(a_5\beta_{1}+\beta_{2})}{a_{10}}e_5,\\[1mm]
\mu(e_4,e_4)&=&\frac{a_5^2(a_{10}\alpha_{1}-a_4\beta_{1})+2a_5(a_{10}\alpha_{2}-a_4\beta_{2})+a_{10}\alpha_{3}-a_4\beta_{3}}{a_1a_{10}}e_1+
\frac{a_5^2\beta_{1}+2a_5\beta_{2}+\beta_{3}}{a_{10}}e_5.\end{array}$$

Now we consider following cases:
\begin{itemize}
\item Let $\beta_{1} \neq0,$ then choosing $a_4=\frac{a_{10}\alpha_{1}}{\beta_{1}},$
    $a_5=-\frac{\beta_{2}}{\beta_{1}},$ $a_{10}=a_1^2\beta_{1},$ we get that $$\mu(e_2,e_2)= e_5, \quad \mu(e_2,e_4)= \delta_{1}e_1, \quad \mu(e_4,e_4)= \frac{\delta_{2}}{a_1}e_1+\frac{\delta_{3}}{a_1^2} e_5,$$
    where $\delta_{1},$ $\delta_{2},$ $\delta_{3}$ are new parameters, which depends on $\alpha_{1}, \alpha_{2}, \alpha_{3}, \beta_{1}, \beta_{2}, \beta_{3}.$
\begin{itemize}
	\item If $\delta_{3} \neq 0,$ then choosing $a_1=\sqrt{\delta_{3}},$ we have the algebra
$\mathcal{L}_{26}^{\alpha, \beta}.$

	\item If $\delta_{3} = 0,$ $\delta_{2}\neq0,$ then choosing $a_1=\delta_{2},$ we have the algebra $\mathcal{L}_{27}^{\alpha}.$
    \item  If $\delta_{2} = 0,$ $\delta_{3}=0,$ then we get the algebra $\mathcal{L}_{28}^{\alpha}.$
\end{itemize}

\item Let $\beta_{1}=0,$ $\beta_{2}\neq0,$ then choosing $a_4=-\frac{a_5a_{10}\alpha_1+a_{10}\alpha_2}{\beta_{2}},$ $a_5=-\frac{\beta_{3}}{2\beta_{2}},$
$a_{10}=a_{1}\beta_{2},$
we get that $$\mu(e_2,e_2)= a_1\alpha_{1}e_1,\quad \mu(e_2,e_4)= e_5, \quad \mu(e_4,e_4)= \frac{\delta} {a_1}e_1.$$
\begin{itemize}
	\item If $\alpha_{1} \neq 0,$ then choosing $a_1=\frac{1}{\alpha_{1}},$ we have the algebra $\mathcal{L}_{29}^{\alpha}.$
		
	\item If $\alpha_{1} = 0,$ $\delta \neq 0,$ then choosing $a_1=\delta$, we have the algebra $\mathcal{L}_{30}.$

\item If $\alpha_{1} = 0,$ $\delta = 0,$ then we obtain  the algebra $\mathcal{L}_{31}.$

\end{itemize}

\item Let $\beta_{1}=0,$ $\beta_{2}=0,$ then $\beta_3\neq0$ and choosing $a_4 =\frac{a_5^2a_{10}\alpha_1 + 2a_5a_{10}\alpha_2}{\beta_3},$ $a_{10}=\beta_{3},$ we have $$\mu(e_2,e_2)= a_1\alpha_{1}e_1,\quad \mu(e_2,e_4)= (a_5\alpha_1+\alpha_2)e_1, \quad \mu(e_4,e_4)= e_5.$$
\begin{itemize}
	\item If $\alpha_{1} \neq 0,$ then choosing $a_1=\frac{1}{\alpha_{1}},$ we have the algebra $\mathcal{L}_{32}.$
		
	\item If $\alpha_{1} = 0,$ then we obtain the algebra $\mathcal{L}_{33}^{\alpha}.$
\end{itemize}
\end{itemize}

Now, we consider the second subcase, i.e., the symmetric bilinear form is  $$\omega(e_2,e_2)=\alpha_{1}e_1,\quad \omega(e_2,e_4)=\alpha_{2}e_1,\quad \omega(e_4,e_4)=\alpha_{3}e_1,$$
$$\omega(e_2,e_5)=\alpha_{4}e_1,\quad \omega(e_4,e_5)=\alpha_{5}e_1,\quad \omega(e_5,e_5)=\alpha_{6}e_1,$$
where $(\alpha_{4}, \alpha_{5},\alpha_{6})\neq (0, 0 , 0).$

Then we have the following restriction
$$\begin{array}{llll}\mu(e_2,e_2)&=&\frac{a_1^2\alpha_{1} + 2a_1a_8\alpha_{4} + a_8^2\alpha_{6}}{a_1}e_1,\\[1mm]
\mu(e_2,e_4)&=&\frac{a_1a_5\alpha_{1} + a_1\alpha_{2} + (a_1a_9+a_5a_8)\alpha_{4} + a_8\alpha_{5} + a_8a_9\alpha_{6}}{a_1}e_1,\\[1mm]
\mu(e_2,e_5)&=&\frac{(a_1\alpha_{4} + a_8\alpha_{6})a_{10}}{a_1}e_1,\\[1mm]
\mu(e_4,e_4)&=&\frac{a_5^2\alpha_{1} + 2a_5\alpha_{2} + \alpha_{3} + 2a_5a_9\alpha_{4} + 2a_9\alpha_{5} + a_9^2\alpha_{6}}{a_1}e_1,\\[1mm]
\mu(e_4,e_5)&=&\frac{(a_5\alpha_{4} + \alpha_{5} + a_9\alpha_{6})a_{10}}{a_1}e_1,\\[1mm]
\mu(e_5,e_5)&=&\frac{a_{10}^2\alpha_{6}}{a_1}e_1.\end{array}$$

\begin{itemize}
\item Let $\alpha_{6} \neq0,$ then choosing $a_{10}=\sqrt{\frac{a_1}{\alpha_{6}}},$ $a_8=-\frac{a_1\alpha_{4}}{\alpha_{6}}$ and
    $a_9=-\frac{a_5\alpha_{4} + \alpha_{5}}{\alpha_{6}},$ we get that $$\mu(e_2,e_2)= a_1\delta_{1}e_1,\quad \mu(e_2,e_4)= (a_5\delta_{1}+\delta_{2})e_1,\quad \mu(e_4,e_4)=\frac{a_5^2\delta_{1} + 2a_5\delta_{2} + \delta_{3}}{a_1}e_1,$$
    $$\mu(e_2,e_5)= 0, \quad \mu(e_4,e_5)= 0, \quad \mu(e_5,e_5)= e_1.$$

\begin{itemize}
	\item If $\delta_{1} \neq 0,$ then choosing $a_1=\frac{1}{\delta_{1}},$ $a_5=-\frac{\delta_{2}}{\delta_{1}},$ we obtain the algebra $\mathcal{L}_{34}^{\alpha}.$

	\item If $\delta_{1} = 0,$ $\delta_{2}\neq0,$ then choosing $a_5=-\frac{\delta_{3}}{2\delta_{2}},$ we have the algebra $\mathcal{L}_{35}^{\alpha}.$

\item If $\delta_{1} = 0,$ $\delta_{2}=0,$ then in the case of $\delta_{3}=0,$ we have the algebra
$\mathcal{L}_{35}^{\alpha=0}$ and in the case of $\delta_{3}\neq 0,$
choosing $a_1=\delta_{3},$ we have the algebra $\mathcal{L}_{36}.$

\end{itemize}

\item Let $\alpha_{6}=0,$ $\alpha_{4}\neq0,$ then choosing $a_{10}=\frac{1}{\alpha_{4}},$ $a_5=-\frac{\alpha_{5}}{\alpha_{4}},$ $a_8=-\frac{a_1\alpha_{1}}{2\alpha_{4}},$ $a_9=\frac{\alpha_{1}\alpha_{5} - \alpha_{2}\alpha_{4}}{\alpha_{4}^2},$ we get
     $$\begin{array}{lll}\mu(e_2,e_2)= 0,& \mu(e_2,e_4)= 0,& \mu(e_4,e_4)=\frac{\delta}{a_1}e_1,\\[1mm]
    \mu(e_2,e_5)= e_1, & \mu(e_4,e_5)= 0, & \mu(e_5,e_5)= 0.\end{array}$$
         Hence, in this case, we obtain the algebras  $\mathcal{L}_{37}$ and  $\mathcal{L}_{38}$
     depending on whether $\delta=0$ or not.

\end{itemize}

\begin{itemize}
\item Let $\alpha_{6}=0,$ $\alpha_{4}=0,$ then $\alpha_{5}\neq0$ and  choosing $a_8=-\frac{a_1a_5\alpha_{1} + a_1\alpha_{2}}{\alpha_{5}},$ $a_9= -\frac{a_5^2\alpha_{1} + 2a_5\alpha_{2} + \alpha_{3}}{2\alpha_{5}},$ $a_{10}=\frac{a_1}{\alpha_{5}},$ we get that
     $$\begin{array}{lll}\mu(e_2,e_2)= a_1 \alpha_1e_1, & \mu(e_2,e_4)= 0, &\mu(e_4,e_4)=0,\\[1mm]
     \mu(e_2,e_5)= 0, & \mu(e_4,e_5)= e_1, & \mu(e_5,e_5)= 0.\end{array}$$
    Hence, in this case, we obtain the algebras  $\mathcal{L}_{39}$ and  $\mathcal{L}_{40}$
     depending on whether $\alpha_1=0$ or not.

\end{itemize}

\textbf{Case 2.} Let $\mathcal{L}$ be a five-dimensional complex solvable symmetric Leibniz algebra, whose underlying Lie algebra is
$$\begin{array}{lllllll}\mathcal{S}_{4.2}\oplus \mathbb{C}: & [e_1,e_4]=e_1, & [e_2,e_4]=e_1+e_2, & [e_3,e_4]=e_2+e_3.\end{array}$$

 Since $Z(\mathcal{S}_{4.2}\oplus \mathbb{C}) = \{e_5\},$ then doing straightforward computations,
we get that the corresponding symmetric bilinear form
satisfying the equation \eqref{eq1.1} is $$\omega(e_4,e_4)=\alpha e_5.$$

Hence, we obtain the algebra $\mathcal{L}_{41}.$

\textbf{Case 3.} Let $\mathcal{L}$ be a five-dimensional complex solvable symmetric Leibniz algebra, whose underlying Lie algebra is
$$\begin{array}{lllllll}\mathcal{S}_{4.3}\oplus \mathbb{C}: & [e_1,e_4]=e_1, & [e_2,e_4]=a e_2, & [e_3,e_4]= b e_3. \end{array}$$

Since $Z(\mathcal{S}_{4.3}\oplus \mathbb{C}) = \{e_5\},$ then by straightforward computations,
we get that the corresponding symmetric bilinear form
satisfying the equation \eqref{eq1.1} is
$$\omega(e_4,e_4)=\alpha e_5.$$

Thus, in this case, we obtain the algebra $\mathcal{L}_{42}^{a, b}.$

\textbf{Case 4.} Let $\mathcal{L}$ be a  five-dimensional complex solvable symmetric Leibniz algebra, whose underlying Lie algebra is
 $$\begin{array}{lllllll}\mathcal{S}_{4.4}\oplus \mathbb{C}: & [e_1,e_4]=e_1, & [e_2,e_4]= e_1+ e_2, & [e_3,e_4]= a e_3. \end{array}$$

Since $Z(\mathcal{S}_{4.4}\oplus \mathbb{C}) = \{e_5\},$ then by straightforward computations,
we get that the corresponding symmetric bilinear form
satisfying the equation \eqref{eq1.1} is
$$\omega(e_4,e_4)=\alpha e_5.$$

Thus, in this case, we obtain the algebra $\mathcal{L}_{43}^{a}.$

\textbf{Case 5.}  Let $\mathcal{L}$ be a complex five-dimensional solvable symmetric Leibniz algebra, whose underlying Lie algebra is
$$\begin{array}{lllllll}\mathcal{S}_{4.6}\oplus \mathbb{C}: & [e_2,e_3]=e_1, & [e_2,e_4]=e_2, & [e_3,e_4]=-e_3.\end{array}$$

Then we have $Z(\mathcal{S}_{4.6}\oplus \mathbb{C}) = \{e_1,e_5\}$ and by straightforward computations,
we get that the corresponding symmetric bilinear form
satisfying the equation \eqref{eq1.1} has the form
$$\omega(e_4,e_4)=\alpha e_1+\beta e_5, \quad \beta\neq 0$$
or
$$\omega(e_4,e_4)=\alpha_1 e_1,  \quad \omega(e_4,e_5)=\alpha_2 e_1, \quad \omega(e_5,e_5)=\alpha_3 e_1, \quad (\alpha_2, \alpha_3)\neq (0, 0).$$

Since the matrix form of the group of automorphisms of the algebra $\mathcal{S}_{4.6}\oplus \mathbb{C}$ is
$$T= \left(
\begin{array}{ccccc}
a_{1}&a_{2}&a_{3}&a_{4}&a_{5}\\
0&a_{6}&0&a_{7}&0\\
0&0&a_{8}&a_{9}&0\\
0&0&0&1&0\\
0&0&0&a_{10}&a_{11}
\end{array}\right),$$
in the first case, we have the restriction
$$\mu(e_4,e_4)=\frac{a_{11}\alpha - a_{5}\beta}{a_{1}a_{11}}e_1+\frac{\beta}{a_{11}}e_5.$$

Then choosing $a_{11}=\beta,$ $a_{5}=\frac{a_{11}\alpha}{\beta},$ we get that $\mu(e_4,e_4)=e_5$ and obtain the algebra $\mathcal{L}_{44}.$

In the second case, we have the restriction
$$\mu(e_4,e_4)=\frac{\alpha_1 + 2a_{10}\alpha_2 + a_{10}^2\alpha_3}{a_{1}}e_1, \quad \mu(e_4,e_5)=\frac{a_{11}(\alpha_2 + a_{10}\alpha_3)}{a_{1}}e_1, \quad
\mu(e_5,e_5)=\frac{a_{11}^2\alpha_3}{a_{1}}e_1.$$

\begin{itemize}
\item Let $\alpha_{3} \neq0,$ then choosing $a_{11}=\sqrt{\frac{a_{1}}{\alpha_{3}}},$ $a_{10}= -\frac{\alpha_{2}}{\alpha_{3}},$
  we get that $\mu(e_4,e_4)= \frac{\alpha}{a_{1}}e_1,$ $\mu(e_4,e_5)= 0$ and $\mu(e_5,e_5)= e_1.$ In this case, we have the algebras $\mathcal{L}_{45}$ and $\mathcal{L}_{46}$ depending on whether $\alpha=0$ or not.

\item Let $\alpha_{3}=0,$ then $\alpha_{2} \neq 0,$ and $a_{10}=-\frac{\alpha_{1}}{2\alpha_{2}},$ we get that
    $\mu(e_4,e_4)=0$ and obtain the algebra $\mathcal{L}_{47}.$

\end{itemize}

\textbf{Case 6.}
Let $\mathcal{L}$ be a five-dimensional complex solvable symmetric Leibniz algebra, whose underlying Lie algebra is
 $$\begin{array}{lllllll}\mathcal{S}_{4.8}\oplus \mathbb{C}: & [e_2, e_3]=e_1, & [e_1,e_4]=(a+1)e_1, & [e_2, e_4]= e_2, & [e_3,e_4]= a e_3. \end{array}$$

Since $Z(\mathcal{S}_{4.8}\oplus \mathbb{C}) = \{e_5\},$ then by straightforward computations,
we get that the corresponding symmetric bilinear form
satisfying the equation \eqref{eq1.1} is
$$\omega(e_4,e_4)=\alpha e_5.$$

Thus, in this case, we obtain the algebra $\mathcal{L}_{48}^{a}.$

\textbf{Case 7.} Let $\mathcal{L}$ be a five-dimensional complex  solvable symmetric Leibniz algebra, whose underlying Lie algebra is
 $$\begin{array}{lllllll}\mathcal{S}_{4.10}\oplus \mathbb{C}: & [e_2, e_3]=e_1, & [e_1,e_4]=2e_1, & [e_2, e_4]= e_2, & [e_3,e_4]= e_2+ e_3. \end{array}$$

Since $Z(\mathcal{S}_{4.10}\oplus \mathbb{C}) = \{e_5\},$ then by straightforward computations, we get that the corresponding symmetric bilinear form
satisfying the equation \eqref{eq1.1} is
$$\omega(e_4,e_4)=\alpha e_5.$$

Thus, in this case we obtain the algebra $\mathcal{L}_{49}.$

\textbf{Case 8.} Let $\mathcal{L}$ be a five-dimensional complex solvable symmetric Leibniz algebra, whose underlying Lie algebra is
 $$\begin{array}{lllllll}\mathcal{S}_{4.11}\oplus \mathbb{C}: & [e_2, e_3]=e_1, & [e_1,e_4]=e_1, & [e_2, e_4]= e_2.\end{array}$$

Since $Z(\mathcal{S}_{4.11}\oplus \mathbb{C}) = \{e_5\},$ then by straightforward computations, we get that the corresponding symmetric bilinear form
satisfying the equation \eqref{eq1.1} is
$$\omega(e_3,e_3)=\alpha e_5,\quad \omega(e_3,e_4)=\beta e_5,\quad \omega(e_4,e_4)=\gamma e_5,$$
where $(\alpha,\beta,\gamma)\neq (0, 0, 0).$

Thus, we have the following class of symmetric Leibniz algebras
\begin{longtable}{ll lllllllllllllll}
$\mathcal{L}_{\omega}$ & $:$ & $e_3\cdot e_2=e_1,$ & $e_1\cdot e_4=e_1,$ & $e_2\cdot e_4=e_2,$ &  $e_3\cdot e_3=\alpha e_5,$ &  $e_3\cdot e_4= \beta e_5,$  \\
& & $e_2 \cdot e_3=- e_1,$ & $e_4\cdot e_1=-e_1,$ & $e_4\cdot e_2=-e_2,$ & $e_4\cdot e_3=\beta e_5,$ & $e_4\cdot e_4= \gamma e_5.$
\end{longtable}

 Since the matrix form of the group of automorphisms of the algebra $\mathcal{S}_{4.11}\oplus \mathbb{C}$ is
$$T= \left(
\begin{array}{ccccc}
a_{3}a_{5}&a_{1}&-a_{4}a_{5}&a_{2}&0\\
0&a_{3}&0&a_{4}&0\\
0&0&a_5&0&0\\
0&0&0&1&0\\
0&0&a_{6}&a_{7}&a_{8}
\end{array}\right),$$
we have the restriction
$$\mu(e_3,e_3)=\frac{\alpha a_{5}^2} {a_{8}}  e_5, \quad \mu(e_3,e_4)=\frac{\beta a_{5}} {a_{8}}  e_5, \quad \mu(e_4,e_4)= \frac{\gamma} {a_{8}}  e_5.$$

Now, we consider the following cases:
\begin{itemize}
\item Let $\alpha = \beta = 0$ and $\gamma \neq 0,$ then choosing $a_{8}=\gamma,$
 we have the algebra $\mathcal{L}_{50}.$

\item Let $\alpha =0,$ $ \beta \neq 0$ and $\gamma = 0,$ then choosing $a_{8}=a_{5}\beta,$ we obtain the algebra $\mathcal{L}_{51}.$

\item Let $\alpha =0,$ $ \beta \neq 0$ and $\gamma \neq 0,$ then choosing $a_{5}=\frac {\gamma}{\beta},$ $a_{8}= \gamma,$
we obtain the algebra $\mathcal{L}_{52}.$

\item Let $\alpha \neq0,$ $ \beta = 0$ and $\gamma = 0,$ then choosing $a_{8}=\alpha a_{5}^2,$
we obtain the algebra $\mathcal{L}_{53}.$

\item Let $\alpha \neq0,$ $ \beta = 0$ and $\gamma \neq 0,$ then choosing $a_{8}=\gamma,$
$a_{5}=\sqrt{\frac {\gamma}{\alpha}},$
we obtain the algebra
$\mathcal{L}_{54}.$

\item Let $\alpha \neq0,$ $ \beta \neq 0,$ then choosing $a_{5}=\frac {\beta}{\alpha},$
$a_{8}=\frac {\beta^2}{\alpha},$
we obtain the algebra $\mathcal{L}_{55}^{\alpha}.$
\end{itemize}

\textbf{Case 9.} Let $\mathcal{L}$ be a five-dimensional complex solvable symmetric Leibniz algebra, whose underlying Lie algebra is
 $$\begin{array}{lllllll}\mathcal{S}_{4.12}\oplus \mathbb{C}: & [e_1, e_3]=e_1, & [e_2,e_4]=e_2.\end{array}$$

Then $Z(\mathcal{S}_{4.12}\oplus \mathbb{C}) = \{e_5\}$ and
$$\omega(e_3,e_3)=\alpha e_5,\quad \omega(e_3,e_4)=\beta e_5,\quad \omega(e_4,e_4)=\gamma e_5,$$
where $(\alpha,\beta,\gamma)\neq (0, 0, 0).$

Thus, we have the following class of symmetric Leibniz algebras
\begin{longtable}{ll lllllllllllllll}
$\mathcal{L}_{\omega}$ & $:$ & $e_1\cdot e_3=e_1,$ & $e_2\cdot e_4=e_2,$ &  $e_3\cdot e_3=\alpha e_5,$ &  $e_3\cdot e_4= \beta e_5,$  \\
& & $e_3\cdot e_1=-e_1,$ & $e_4\cdot e_2=-e_2,$ & $e_4\cdot e_3=\beta e_5,$ & $e_4\cdot e_4= \gamma e_5.$
\end{longtable}

 Since the matrix form of the group of automorphisms of the algebra $\mathcal{S}_{4.12}\oplus \mathbb{C}$ is
$$T= \left(
\begin{array}{ccccc}
a_{1}&0& a_{2}& 0 &0\\
0&a_{3}&0&a_{4}&0\\
0&0&1&0&0\\
0&0&0&1&0\\
0&0&a_{5}&a_{6}&a_{7}
\end{array}\right),$$
we have the restriction
$$\mu(e_3,e_3)=\frac{\alpha}{a_{7}}e_5, \quad \mu(e_3,e_4)=\frac{\beta}{a_{7}}e_5, \quad \mu(e_4,e_4)=\frac{\gamma}{a_{7}}e_5.$$

From this restriction, we get that the first non-vanishing parameter of $(\alpha,\beta,\gamma)$ can be scaled to 1 and obtain the algebras $\mathcal{L}_{56},$ $\mathcal{L}_{57}^{\alpha},$ $\mathcal{L}_{58}^{\alpha, \beta}.$
\end{proof}

\subsection{Five-dimensional solvable symmetric Leibniz algebras, whose underlying Lie algebra is split with more than one-dimensional direct factor}

It should be noted that when the dimension of the
center of the underlying Lie algebra is bigger than two,
then to get non-isomorphic algebras from the given family using the Proposition \ref{prop2}
is more difficult and the calculation increases significantly. Therefore, in this subsection, we also use the standard basis changing method and obtain the complete classification of five-dimensional solvable symmetric Leibniz algebras, when the underlying Lie algebra is split with more than one-dimensional direct factor.

First, we consider the case, when the underlying Lie algebra is a direct sum of three-dimensional non-split algebra and two-dimensional abelian ideal, i.e., $\mathcal{S}_{3} \oplus \mathbb{C}^2$.

\begin{theorem} \label{thm2.3} Let $\mathcal{L}$ be a complex five-dimensional solvable non-split symmetric Leibniz algebra, whose underlying Lie algebra is $\mathcal{S}_{3} \oplus \mathbb{C}^2.$ Then it is isomorphic to one of the following pairwise non-isomorphic algebras
 \begin{longtable}{ll lllllllllllllll}
 \hline
$\mathcal{L}_{59}$ & $:$ & $e_1\cdot e_3=e_1,$ & $e_2\cdot e_3=e_1+e_2,$ &  $e_3\cdot e_3= e_4,$ \\
& & $e_3\cdot e_1=-e_1,$ & $e_3\cdot e_2=-e_1-e_2,$ & $e_3\cdot e_5= e_5,$ & $e_5\cdot e_3= e_5.$\\ \hline
$\mathcal{L}_{60}$ & $:$ & $e_1\cdot e_3=e_1,$ & $e_2\cdot e_3=e_1+e_2,$ &  $e_3\cdot e_4= e_5,$ \\
& & $e_3\cdot e_1=-e_1,$ & $e_3\cdot e_2=-e_1-e_2,$ & $e_4\cdot e_3= e_5.$ \\ \hline
$\mathcal{L}_{61}$ & $:$ & $e_1\cdot e_3=e_1,$ & $e_2\cdot e_3=e_1+e_2,$ &  $e_3\cdot e_3= e_5,$ \\
& & $e_3\cdot e_1=-e_1,$ & $e_3\cdot e_2=-e_1-e_2,$ & $e_4\cdot e_4= e_5.$\\ \hline
$\mathcal{L}_{62}^{a\neq0}$ & $:$ & $e_1\cdot e_3=e_1,$ & $e_2\cdot e_3=a e_2,$ &  $e_3\cdot e_3= e_4,$ \\
& & $e_3\cdot e_1=-e_1,$ & $e_3\cdot e_2=-ae_2,$ & $e_3\cdot e_5= e_5,$ & $e_5\cdot e_3= e_5.$\\ \hline
$\mathcal{L}_{63}^{a\neq0}$ & $:$ & $e_1\cdot e_3=e_1,$ & $e_2\cdot e_3=a e_2,$ &  $e_3\cdot e_4= e_5,$ \\
& & $e_3\cdot e_1=-e_1,$ & $e_3\cdot e_2=-a e_2,$ & $e_4\cdot e_3= e_5.$ \\ \hline
$\mathcal{L}_{64}^{a\neq0}$ & $:$ & $e_1\cdot e_3=e_1,$ & $e_2\cdot e_3= a e_2,$ &  $e_3\cdot e_3= e_5,$ \\
& & $e_3\cdot e_1=-e_1,$ & $e_3\cdot e_2=-a e_2,$ & $e_4\cdot e_4= e_5.$\\ \hline
\end{longtable}

\end{theorem}

\begin{proof}
Since any  three-dimensional solvable Lie algebra is isomorphic to one of the algebras
$$\begin{array}{lllllll}\mathcal{S}_{3.1}: & [e_1,e_3]=e_1, & [e_2,e_3]=e_1+e_2, \\[1mm]
\mathcal{S}_{3.2}: & [e_1,e_3]=e_1, & [e_2,e_3]=ae_2, & a\neq 0,
\end{array}$$
we consider following two cases.

 Let $\mathcal{L}$ be a complex five-dimensional solvable symmetric Leibniz algebra, whose underlying Lie algebra is $\mathcal{G}=\mathcal{S}_{3.1} \oplus \mathbb{C}^2.$
Since the center of the Lie algebra $\mathcal{G}$ is $\{e_4,e_5\},$ we consider symmetric bilinear form
$\omega: \mathcal{G}\times \mathcal{G}\rightarrow \{e_4,e_5\}.$ Put
$$\omega(e_i,e_j)=\alpha_{i,j}e_4+\beta_{i,j}e_5.$$

From the condition $\omega([e_i,e_j],e_k)=0,$ we obtain that $\omega(e_1,e_i)=\omega(e_2,e_i)=0$ for $1\leq i \leq 5.$
Thus, we have that the symmetric Leibniz algebra corresponding to the Lie algebra $\mathcal{S}_{3.1} \oplus \mathbb{C}^2$ has a multiplication
$$\left\{\begin{array}{lllll} e_1  \cdot e_3=e_1, &  e_2 \cdot e_3=e_1+e_2,\quad \\ e_3 \cdot e_1=-e_1, & e_3  \cdot e_2=-e_1-e_2,\\[1mm]
e_i \cdot e_j =\omega(e_i,e_j)=\alpha_{i,j}e_4+\beta_{i,j}e_5, & 3\leq i \leq j\leq 5. \end{array}\right. \eqno (5) $$

From this multiplication we have that the linear space $\mathcal{V}=\{e_4,e_5\}$ is a two-dimensional ideal of symmetric Leibniz algebra. Since any two-dimensional symmetric Leibniz algebra is an abelian or isomorphic to the algebra $\lambda_2$ with multiplication $x \cdot  x = y,$ we can consider following subcases.

\begin{itemize}
\item Let $\mathcal{V}$ be an abelian algebra, then we have $$\omega(e_4,e_4)=\omega(e_4,e_5)=\omega(e_5,e_5)=0.$$

Moreover, from the condition $\omega(\omega(e_i, e_j), e_k)=0,$ we have that
$$\left\{\begin{array}{lllll}\alpha_{3,3}\alpha_{3,4}+\beta_{3,3}\alpha_{3,5}=0, & \alpha_{3,3}\beta_{3,4}+\beta_{3,3}\beta_{3,5}=0,\\[1mm]\alpha_{3,4}^2+\beta_{3,4}\alpha_{3,5}=0, & \alpha_{3,4}\beta_{3,4}+\beta_{3,4}\beta_{3,5}=0,\\[1mm]
\alpha_{3,4}\alpha_{3,5}+\beta_{3,5}\alpha_{3,5}=0, & \alpha_{3,5}\beta_{3,4}+\beta_{3,5}^2=0.\end{array}\right.\eqno (6)$$
\begin{itemize}

\item Let $(\alpha_{3,3},\beta_{3,3})\neq(0,0),$ then making the change $e'_4=\alpha_{3,3}e_4+\beta_{3,3}e_5$ in (5), we can suppose $\alpha_{3,3}=1,$ $\beta_{3,3}=0.$ Then from the equalities (6), we have that $\alpha_{3,4}=\beta_{3,4}=\beta_{3,5}=0.$
        If $\alpha_{3,5}=0,$ then we have the split algebra.
                Thus, $\alpha_{3,5}\neq0,$ and choosing $e'_5=\frac{1}{\alpha_{3,5}}e_5,$ we obtain the algebra $\mathcal{L}_{59}.$

\item Let $(\alpha_{3,3},\beta_{3,3})=(0,0).$
    If $\beta_{3,4}=\alpha_{3,5}=0,$ then from (6) it follows that $\alpha_{3,4}=\beta_{3,5}=0,$ and we obtain the split algebra. Thus, we may assume $(\beta_{3,4},\alpha_{3,5})\neq(0,0)$ and taking into account of the symmetrically $e_4$ and $e_5,$ we can suppose $\beta_{3,4}\neq0.$ Then making the change $e_5'=\alpha_{3,4}e_4+\beta_{3,4}e_5,$ we can assume $\alpha_{3,4}=0,$ $\beta_{3,4}=1$ and from (6), we have $\alpha_{3,5}=\beta_{3,5}=0.$ Hence, we obtain the algebra $\mathcal{L}_{60}.$

\end{itemize}

\item  Let $\mathcal{V}$ be isomorphic to the algebra $\lambda_2,$ then we can suppose
$$\omega(e_4,e_4)=e_5, \quad \omega(e_4,e_5)=\omega(e_5,e_5)=0.$$

From the condition $\omega(\omega(e_i,e_j),e_k)=0,$ we have that $\alpha_{3,3}=\alpha_{3,4}=\alpha_{3,5}=\beta_{3,5}=0.$
Moreover, making the change $e'_3=e_3-\beta_{3,4} e_4,$ we can suppose $\beta_{3,4}=0$ and in case of $\beta_{3,3}=0,$ we have the split algebra. Thus,
 $\beta_{3,3}\neq0$ and making the change $$e'_1=\beta_{3,3} e_1, \ e'_2=\beta_{3,3} e_2, \ e'_3= e_3, \ e'_4=\sqrt{\beta_{3,3}} e_4, \ e'_5=\beta_{3,3} e_5,$$ we obtain the algebra $\mathcal{L}_{61}.$

 \end{itemize}

 Similarly, in the case, when the underlying Lie algebra is $\mathcal{S}_{3.2} \oplus \mathbb{C}^2,$ we obtain the algebras $\mathcal{L}_{62}^{a},$ $\mathcal{L}_{63}^{a}$ and $\mathcal{L}_{64}^{a}.$
\end{proof}

Now, we consider five-dimensional complex solvable symmetric Leibniz algebras, whose underlying Lie algebra is $\mathcal{S}_{2.1} \oplus \mathbb{C}^3,$ where $\mathcal{S}_{2.1}$ is a two-dimensional solvable Lie algebra with multiplication $[e_1,e_2]=e_1.$

\begin{theorem} \label{thm2.4} Let $\mathcal{L}$ be a complex five-dimensional solvable non-split symmetric Leibniz algebra, whose underlying Lie algebra is $\mathcal{S}_{2.1} \oplus \mathbb{C}^3.$ Then it is isomorphic to one of the following pairwise non-isomorphic algebras
\begin{longtable}{llllllll} \hline
$\mathcal{L}_{65}$ & $:$ & $e_1\cdot e_2=e_1,$ & $e_3\cdot e_2=e_5,$ & $e_2\cdot e_1=-e_1,$ &  $e_2\cdot e_3=e_5,$& $e_3\cdot e_3=e_4.$ \\ \hline
$\mathcal{L}_{66}$ & $:$ &$e_1\cdot e_2=e_1,$ & $e_3\cdot e_2=e_5,$ & $e_2\cdot e_2=e_4,$\\
&& $e_2\cdot e_1=-e_1,$ &  $e_2\cdot e_3=e_5,$ & $e_3\cdot e_3=e_4.$ \\ \hline
$\mathcal{L}_{67}$ & $:$ & $e_1\cdot e_2=e_1,$ & $e_2\cdot e_5=e_4,$ & $e_3\cdot e_3=e_4,$ &$e_2\cdot e_1=-e_1,$ &  $e_5\cdot e_2=e_4.$ \\ \hline
$\mathcal{L}_{68}$ & $:$ & $e_1\cdot e_2=e_1,$ & $e_3\cdot e_4=e_5,$ & $e_2\cdot e_2=e_5$& $e_2\cdot e_1=-e_1,$ &  $e_4\cdot e_3=e_5.$ \\ \hline
\end{longtable}
\end{theorem}

\begin{proof}
Let $\mathcal{G}=\mathcal{S}_{2.1} \oplus \mathbb{C}^3$ be a Lie algebra with a basis $\{e_1,e_2,e_3,e_4,e_5\},$ such that $[e_1,e_2]=e_1.$
Since the center of the Lie algebra is $\operatorname{span}\{e_3,e_4,e_5\},$ we consider symmetric bilinear form
$\omega: \mathcal{G}\times \mathcal{G}\rightarrow \{e_3,e_4,e_5\}.$ Put
$$\omega(e_i,e_j)=\alpha_{i,j}e_3+\beta_{i,j}e_4+\gamma_{i,j}e_5.$$

From the condition $\omega([e_i,e_j],e_k)=0,$ we obtain that $\omega(e_1,e_i)=0$ for $1\leq i \leq 5.$
Thus, we have that the symmetric Leibniz algebra has a multiplication
$$\left\{\begin{array}{lllll}e_1 \cdot e_2=e_1,\quad e_2 \cdot e_1=-e_1,\\[1mm] e_i \cdot e_j=\omega(e_i, e_j)=\alpha_{i,j}e_3+\beta_{i,j}e_4+\gamma_{i,j}e_5, \quad 2\leq i \leq j\leq 5. \end{array}\right.$$

From this multiplication we have that $\mathcal{V}=\{e_3,e_4,e_5\}$ is a three-dimensional commutative ideal of symmetric Leibniz algebra. It is known that three-dimensional commutative symmetric Leibniz algebra is either abelian or isomorphic to $\lambda_2$ or ${\mathcal N}_1.$

Thus, we can consider following cases.
\begin{itemize}
\item Let $\mathcal{V}$ be an abelian algebra, then we have $$\omega(e_3,e_3)=\omega(e_3,e_4)=\omega(e_3,e_5)=\omega(e_4,e_4)=\omega(e_4,e_5)=\omega(e_5,e_5)=0.$$
Now from the condition $\omega(\omega(e_i,e_j),e_k)=0,$ we have that
$$\left\{\begin{array}{lllll}\alpha_{2,2}\alpha_{2,3}+\beta_{2,2}\alpha_{2,4}+\gamma_{2,2}\alpha_{2,5}=0,
&\alpha_{2,2}\beta_{2,3}+\beta_{2,2}\beta_{2,4}+\gamma_{2,2}\beta_{2,5}=0,\\[1mm]
\alpha_{2,2}\gamma_{2,3}+\beta_{2,2}\gamma_{2,4}+\gamma_{2,2}\gamma_{2,5}=0, & \alpha_{2,3}^2+\beta_{2,3}\alpha_{2,4}+\gamma_{2,3}\alpha_{2,5}=0,\\[1mm] \alpha_{2,3}\beta_{2,3}+\beta_{2,3}\beta_{2,4}+\gamma_{2,3}\beta_{2,5}=0,& \alpha_{2,3}\gamma_{2,3}+\beta_{2,3}\gamma_{2,4}+\gamma_{2,3}\gamma_{2,5}=0,\\[1mm] \alpha_{2,4}\alpha_{2,3}+\beta_{2,4}\alpha_{2,4}+\gamma_{2,4}\alpha_{2,5}=0, & \alpha_{2,4}\beta_{2,3}+\beta_{2,4}^2+\gamma_{2,4}\beta_{2,5}=0,\\[1mm]
\alpha_{2,4}\gamma_{2,3}+\beta_{2,4}\gamma_{2,4}+\gamma_{2,4}\gamma_{2,5}=0,& \alpha_{2,5}\alpha_{2,3}+\beta_{2,5}\alpha_{2,4}+\gamma_{2,5}\alpha_{2,5}=0,\\[1mm]
\alpha_{2,5}\beta_{2,3}+\beta_{2,5}\beta_{2,4}+\gamma_{2,5}\beta_{2,5}=0,& \alpha_{2,5}\gamma_{2,3}+\beta_{2,5}\gamma_{2,4}+\gamma_{2,5}^2=0_.\end{array}\right.\eqno (7)$$
\begin{itemize}
\item  Let $(\alpha_{2,2},\beta_{2,2},\gamma_{2,2})\neq(0,0,0),$ then making the change $e'_3=\alpha_{2,2}e_3+\beta_{2,2}e_4+\gamma_{2,2}e_5,$ we can suppose $\alpha_{2,2}=1,$ $\beta_{2,2}=0,$ $\gamma_{2,2}=0$ and from (7) we have that
$$\left\{\begin{array}{lllll}
\alpha_{2,3}=0, & \beta_{2,3}=0, & \gamma_{2,3}=0, & \\[1mm]
\beta_{2,4}\alpha_{2,4}+\gamma_{2,4}\alpha_{2,5}=0,& \beta_{2,4}^2+\gamma_{2,4}\beta_{2,5}=0,& \beta_{2,4}\gamma_{2,4}+\gamma_{2,4}\gamma_{2,5}=0, \\[1mm] \beta_{2,5}\alpha_{2,4}+\gamma_{2,5}\alpha_{2,5}=0,& \beta_{2,5}\beta_{2,4}+\gamma_{2,5}\beta_{2,5}=0, & \beta_{2,5}\gamma_{2,4}+\gamma_{2,5}^2=0.\end{array}\right.$$

Moreover, taking the suitable basis of the vector space $\operatorname{span}\{e_4, e_5\},$ we can suppose $\beta_{2,5} =0,$ which implies $\beta_{2,4}=\gamma_{2,5}=0$ and $\alpha_{2,5}\gamma_{2,4}=0.$
In the case of $\alpha_{2,5}=0,$ we obtain the split algebra, and in the case of $\alpha_{2,5}\neq 0,$ which implies $\gamma_{2,4}=0,$ making the change $e_4'=e_4 - \frac{\alpha_{2,4}}{\alpha_{2,5}} e_5,$ we again obtain the split algebra.

\item  Let $(\alpha_{2,2},\beta_{2,2},\gamma_{2,2})=(0,0,0),$ then we have that

$$\left\{\begin{array}{lllll} \alpha_{2,3}^2+\beta_{2,3}\alpha_{2,4}+\gamma_{2,3}\alpha_{2,5}=0, & \alpha_{2,3}\beta_{2,3}+\beta_{2,3}\beta_{2,4}+\gamma_{2,3}\beta_{2,5}=0,\\[1mm] \alpha_{2,3}\gamma_{2,3}+\beta_{2,3}\gamma_{2,4}+\gamma_{2,3}\gamma_{2,5}=0,& \alpha_{2,4}\alpha_{2,3}+\beta_{2,4}\alpha_{2,4}+\gamma_{2,4}\alpha_{2,5}=0,\\[1mm] \alpha_{2,4}\beta_{2,3}+\beta_{2,4}^2+\gamma_{2,4}\beta_{2,5}=0,& \alpha_{2,4}\gamma_{2,3}+\beta_{2,4}\gamma_{2,4}+\gamma_{2,4}\gamma_{2,5}=0,\\[1mm] \alpha_{2,5}\alpha_{2,3}+\beta_{2,5}\alpha_{2,4}+\gamma_{2,5}\alpha_{2,5}=0,& \alpha_{2,5}\beta_{2,3}+\beta_{2,5}\beta_{2,4}+\gamma_{2,5}\beta_{2,5}=0,\\[1mm] \alpha_{2,5}\gamma_{2,3}+\beta_{2,5}\gamma_{2,4}+\gamma_{2,5}^2=0_.\end{array}\right.\eqno (8)$$

Here we can consider the operator of $ad_{e_2}$ as a linear map of the vector space $\operatorname{span}\{e_3, e_4, e_5\}.$ Choosing the suitable basis of the vector space
$\operatorname{span}\{e_3, e_4, e_5\},$ we can obtain the Jordan form of the matrix of the operator $ad_{e_2}.$ It means that we can always assume $\alpha_{2,4}=\alpha_{2,5}=\beta_{2,5}=\gamma_{2,3}=0.$
Then from the equality (8), we obtain $\alpha_{2,3}=\beta_{2,4}=\gamma_{2,5}=0$ and $\beta_{2,3}\gamma_{2,4}=0.$
So in this case also, we have the split algebra.

\end{itemize}

\item Let the three-dimensional algebra $\mathcal{V}$ be isomorphic to $\lambda_2$, then we have $$\omega(e_3,e_3)=e_4, \quad \omega(e_3,e_4)=\omega(e_3,e_5)=\omega(e_4,e_4)=\omega(e_4,e_5)=\omega(e_5,e_5)=0.$$

    Then from the condition $\omega(\omega(e_i,e_j),e_k)=0,$ we have that
    $$\omega(e_2, e_2) = \beta_{2,2}e_4+\gamma_{2,2}e_5, \quad \omega(e_2, e_3) = \beta_{2,3}e_4+\gamma_{2,3}e_5,\quad \omega(e_2, e_5) = \beta_{2,5}e_4,$$
    with restriction $\gamma_{2,2}\beta_{2,5}=0,$ $\gamma_{2,3}\beta_{2,5}=0.$

    \begin{itemize}
\item Let $\beta_{2,5}=0,$ then we have the multiplication
$$\left\{\begin{array}{lllll}e_1\cdot e_2=e_1, & e_3\cdot e_2=\beta_{2,3}e_4+\gamma_{2,3}e_5, & e_2\cdot e_2=\beta_{2,2}e_4+\gamma_{2,2}e_5,\\[1mm]
 e_2\cdot e_1=-e_1, &  e_2\cdot e_3=\beta_{2,3}e_4+\gamma_{2,3}e_5,& e_3\cdot e_3=e_4. \end{array}\right.$$

   If $\gamma_{2,3}\neq 0,$ $4\beta_{2,2}\gamma_{2,3}^2-4\beta_{2,3}\gamma_{2,2}\gamma_{2,3}+\gamma_{2,2}^2 \neq0,$ then making the change $$e_1'=e_1, \quad e_2'=e_2 + A e_3, \quad e_3'=Be_3, \quad e_4'=B^2e_3, \quad e_5' = B(A+\beta_{2,3})e_4 + B \gamma_{2,3} e_5,$$
   where $A= - \frac{\gamma_{2,2}}{2\gamma_{2,3}},$ $B = \frac{4\beta_{2,2}\gamma_{2,3}^2-4\beta_{2,3}\gamma_{2,2}\gamma_{2,3}+\gamma_{2,2}^2} {2\gamma_{2,3}},$
     we obtain the algebra $\mathcal{L}_{65}.$

If $\gamma_{2,3}\neq 0,$ $4\beta_{2,2}\gamma_{2,3}^2-4\beta_{2,3}\gamma_{2,2}\gamma_{2,3}+\gamma_{2,2}^2 = 0,$ then making the change $$e_1'=e_1, \quad e_2'=e_2 + A e_3, \quad e_3'=e_3, \quad e_4'=e_4, \quad e_5' = (A+\beta_{2,3})e_4 + \gamma_{2,3} e_5,$$
   where $A= - \frac{\gamma_{2,2}}{2\gamma_{2,3}},$ we have the algebra $\mathcal{L}_{66}.$

If $\gamma_{2,3} =0,$ then $\gamma_{2,2} \neq 0,$ and after the change $$e_1'=e_1, \quad e_2'=e_2 + A e_3, \quad e_3'=e_3, \quad e_4'=e_4, \quad e_5' = (\beta_{2,2}+2A\beta_{2,3}+A^2)e_4 + \gamma_{2,2} e_5,$$
where $A = - \beta_{2,3},$ we find out that this algebra is split.

\item Let $\beta_{2,5}\neq0,$ then $\gamma_{2,2}=\gamma_{2,3}=0$ and making
the change $$e_1'=e_1, \quad e_2'=e_2 - \frac{\beta_{2,2}}{2\beta_{2,5}} e_5, \quad e_3'=e_3- \frac{\beta_{2,3}}{\beta_{2,5}}e_5, \quad e_4'=e_4, \quad e_5' = \frac{1}{\beta_{2,5}}e_5,$$
we have the algebra $\mathcal{L}_{67}.$

\end{itemize}

\item Let the three-dimensional algebra $\mathcal{V}$ be isomorphic to ${\mathcal N}_1$, then we have
    $$\omega(e_3,e_4)=e_5, \quad \omega(e_3,e_3)=\omega(e_3,e_5)=\omega(e_4,e_4)=\omega(e_4,e_5)=\omega(e_5,e_5)=0.$$
Now, from the condition $\omega(\omega(e_i,e_j),e_k)=0,$ we have that
$$\omega(e_2, e_2) = \gamma_{2,2}e_5, \quad \omega(e_2, e_3) = \gamma_{2,3}e_5,\quad \omega(e_2, e_4) = \gamma_{2,4}e_5,$$

Making the basis change
$$e_1'=e_1, \quad e_2'=e_2 - \gamma_{2,4} e_3 - \gamma_{2,3} e_4 , \quad e_3'=e_3, \quad e_4'=e_4, \quad e_5' = e_5,$$
we may suppose $\gamma_{2,3}=\gamma_{2,4}=0.$ Then $\gamma_{2,2} \neq 0$ and we obtain the algebra $\mathcal{L}_{68}.$
\end{itemize}
\end{proof}

Now, we consider the case when, underlying Lie algebra is a direct sum of two copies of two-dimensional algebra $\mathcal{S}_{2.1}$ and  one-dimensional abelian ideal, i.e., the Lie algebra $\mathcal{G}$ has the multiplication
 $$\mathcal{G}: [e_1,e_2]=e_1, \quad [e_3,e_4]=e_3.$$

\begin{theorem} \label{thm2.5} Let $\mathcal{L}$ be a complex five-dimensional solvable non-split symmetric Leibniz algebra, whose underlying Lie algebra is $\mathcal{G}$ $([e_1,e_2]=e_1, [e_3,e_4]=e_3)$, then it is isomorphic to one of the following non-isomorphic algebras
 \begin{longtable}{ll lllllllllllllll}
 \hline
$\mathcal{L}_{69}^{\alpha, \beta}$ & $:$ & $e_1\cdot e_2=e_1,$ & $e_3\cdot e_4=e_3,$ & $e_2\cdot e_2= \alpha e_5,$ & $e_2 \cdot e_4= e_5,$\\
& & $e_2\cdot e_1=-e_1,$ & $e_4\cdot e_3=-e_3,$ & $e_4 \cdot e_2= e_5,$ & $e_4\cdot e_4= \beta e_5.$\\ \hline
$\mathcal{L}_{70}^{\alpha}$ & $:$ & $e_1\cdot e_2=e_1,$ & $e_3\cdot e_4=e_3,$ & $e_2 \cdot e_2=\alpha e_5,$ \\
& & $e_2\cdot e_1=-e_1,$ & $e_4\cdot e_3=-e_3,$ & $e_4 \cdot e_4=e_5.$ \\ \hline
\end{longtable}
where $\mathcal{L}_{69}^{\alpha, \beta} \cong \mathcal{L}_{69}^{\beta, \alpha}$ and
$\mathcal{L}_{70}^{\alpha} \cong \mathcal{L}_{70}^{\frac 1{\alpha}}.$
\end{theorem}
\begin{proof}
Since $Z(\mathcal{G}) = \{e_5\},$ then by straightforward computations, we get that the corresponding symmetric bilinear form
satisfying the equation \eqref{eq1.1} is
$$\omega(e_2, e_2)=\alpha e_5,\quad \omega(e_2, e_4)=\beta e_5,\quad \omega(e_4, e_4)=\gamma e_5,$$
where $(\beta,\alpha\gamma)\neq (0, 0).$

Thus, we have the following class of symmetric Leibniz algebras
\begin{longtable}{ll lllllllllllllll}
$\mathcal{L}_{\omega}$ & $:$ & $e_1\cdot e_2=e_1,$ & $e_3\cdot e_4=e_3,$ &  $e_2\cdot e_2=\alpha e_5,$ &  $e_4\cdot e_4= \gamma e_5,$  \\
& & $e_2\cdot e_1=-e_1,$ & $e_4\cdot e_3=-e_3,$ & $e_2\cdot e_4=\beta e_5,$ & $e_4\cdot e_2= \beta e_5.$
\end{longtable}

 Since the matrix form of the group of automorphisms of the algebra $\mathcal{G}$ is
$$T_1= \left(
\begin{array}{ccccc}
a_{1}&a_{2}&0& 0 &0\\
0&1&0&0&0\\
0&0&a_{3}&a_{4}&0\\
0&0&0&1&0\\
0&a_{5}&0&a_{6}&a_{7}
\end{array}\right) \quad \text{or} \quad T_2= \left(
\begin{array}{ccccc}
0&0&a_{3}&a_{4}&0\\
0&0&0&1&0\\
a_{1}&a_{2}&0& 0 &0\\
0&1&0&0&0\\
0&a_{5}&0&a_{6}&a_{7}
\end{array}\right),$$
we have the restriction
$$\mu(e_2,e_2)=\frac{\alpha}{a_{7}}e_5, \quad \mu(e_2,e_4)=\frac{\beta}{a_{7}}e_5, \quad \mu(e_4,e_4)=\frac{\gamma}{a_{7}}e_5$$
or
$$\mu(e_2,e_2)=\frac{\gamma}{a_{7}}e_5, \quad \mu(e_2,e_4)=\frac{\beta}{a_{7}}e_5, \quad \mu(e_4,e_4)=\frac{\alpha}{a_{7}}e_5.$$


If $\beta \neq 0,$ then choosing $a_7=\beta,$ we get the algebra $\mathcal{L}_{69}^{\alpha, \beta}.$

If $\beta = 0,$ then $\alpha \gamma =0$ and choosing $a_7=\gamma,$ we get the algebra $\mathcal{L}_{70}^{\alpha}.$

\end{proof}

Finally, we consider the case when, underlying Lie algebra is a direct sum of two-dimensional solvable algebra $\mathcal{S}_{2.1}$ and three-dimensional Heisenberg algebra, i.e., the Lie algebra $\mathcal{G}$ has the multiplication
 $$\mathcal{G}: [e_1, e_2]=e_1, \quad [e_3, e_4]=e_5.$$

\begin{theorem} \label{thm2.6} Let $\mathcal{L}$ be a complex five-dimensional solvable non-split symmetric Leibniz algebra, whose underlying Lie algebra is $\mathcal{G}$ $([e_1,e_2]=e_1, [e_3,e_4]=e_5)$, then it is isomorphic to one of the following pairwise non-isomorphic algebras
 \begin{longtable}{ll lllllllllllllll}
 $\mathcal{L}_{71}^{\alpha}$ & $:$ & $e_1\cdot e_2=e_1,$ & $e_3\cdot e_4=(\alpha+1)e_5,$ & $e_2\cdot e_2= e_5,$\\
& & $e_2\cdot e_1=-e_1,$ & $e_4\cdot e_3=(\alpha-1)e_5.$ \\ \hline
$\mathcal{L}_{72}^{\alpha}$ & $:$ & $e_1\cdot e_2=e_1,$ & $e_3\cdot e_4=(\alpha+1)e_5,$ & $e_2 \cdot e_4=e_5,$\\
& & $e_2\cdot e_1=-e_1,$ & $e_4\cdot e_3=(\alpha-1)e_5,$ & $e_4 \cdot e_2=e_5.$\\ \hline
$\mathcal{L}_{73}^{\alpha}$ & $:$ & $e_1\cdot e_2=e_1,$ & $e_3\cdot e_4=(\alpha+1)e_5,$ & $e_2 \cdot e_4=e_5,$& $e_2\cdot e_2= e_5,$\\
& & $e_2\cdot e_1=-e_1,$ & $e_4\cdot e_3=(\alpha-1)e_5,$ & $e_4 \cdot e_2=e_5.$\\ \hline
$\mathcal{L}_{74}^{\alpha\neq0, \beta}$ & $:$ & $e_1\cdot e_2=e_1,$ & $e_3\cdot e_4=(\alpha+1)e_5,$ & $e_2 \cdot e_3=e_5,$ & $e_2 \cdot e_4= e_5,$ & $e_2 \cdot e_2= \beta e_5.$\\
& & $e_2\cdot e_1=-e_1,$ & $e_4\cdot e_3=(\alpha-1)e_5,$ & $e_3 \cdot e_2=e_5,$ & $e_4 \cdot e_2= e_5.$\\ \hline
$\mathcal{L}_{75}$ & $:$ & $e_1\cdot e_2=e_1,$ & $e_3\cdot e_4=e_5,$ & $e_2\cdot e_2= e_5,$ \\
& & $e_2\cdot e_1=-e_1,$ & $e_4\cdot e_3=-e_5$ & $e_4\cdot e_4= e_5.$\\ \hline
$\mathcal{L}_{76}^{\alpha}$ & $:$ & $e_1\cdot e_2=e_1,$ & $e_3\cdot e_4=e_5,$ & $e_2\cdot e_2= \alpha e_5,$ & $e_2\cdot e_3= e_5,$\\
& & $e_2\cdot e_1=-e_1,$ & $e_4\cdot e_3=-e_5$ & $e_4\cdot e_4= e_5,$ & $e_3\cdot e_2= e_5.$\\ \hline
$\mathcal{L}_{77}^{\alpha}$ & $:$ & $e_1\cdot e_2=e_1,$ & $e_3\cdot e_4=e_5,$ & $e_2\cdot e_2= \alpha e_5,$ & $e_2\cdot e_4= e_5,$\\
& & $e_2\cdot e_1=-e_1,$ & $e_4\cdot e_3=-e_5$ & $e_4\cdot e_4= e_5,$ & $e_4\cdot e_2= e_5.$\\ \hline
\end{longtable}
where $\mathcal{L}_{74}^{\alpha, \beta} \cong \mathcal{L}_{74}^{-\alpha, -\beta}.$
\end{theorem}
\begin{proof}
Since $Z(\mathcal{G}) = \{e_5\},$ then by straightforward computations, we get that the corresponding symmetric bilinear form
satisfying the equation \eqref{eq1.1} is
$$\omega(e_2,e_2)=\alpha_1 e_5,\quad \omega(e_2,e_3)=\alpha_2 e_5,\quad \omega(e_2,e_4)=\alpha_3 e_5,$$
$$\omega(e_3,e_3)=\alpha_4 e_5,\quad \omega(e_3,e_4)=\alpha_5 e_5,\quad \omega(e_4,e_4)=\alpha_6 e_5,$$
where $(\alpha_1,\alpha_2, \alpha_3 )\neq (0, 0, 0).$

Thus, we have the following class of symmetric Leibniz algebras
\begin{longtable}{ll lllllllllllllll}
$\mathcal{L}_{\omega}$ & $:$ & $e_1\cdot e_2=e_1,$ & $e_3\cdot e_4=(\alpha_5+1)e_5,$ &  $e_2\cdot e_3=\alpha_2 e_5,$ &  $e_2\cdot e_4= \alpha_3 e_5,$  \\
& & $e_2\cdot e_1=-e_1,$ & $e_4\cdot e_3=(\alpha_5-1)e_5,$ & $e_3\cdot e_2=\alpha_2 e_5,$ & $e_4\cdot e_2= \alpha_3 e_5,$ \\
& & $e_2 \cdot e_2=\alpha_1 e_5,$ & $e_3 \cdot e_3=\alpha_4 e_5,$ & $e_4 \cdot e_4=\alpha_6 e_5.$
\end{longtable}

 Since the matrix form of the group of automorphisms of the algebra $\mathcal{G}$ is
$$T= \left(
\begin{array}{ccccc}
a_{1}&a_{2}&0& 0 &0\\
0&1&0&0&0\\
0&0&a_{3}&a_{4}&0\\
0&0&a_{5}&a_{6}&0\\
0&a_{7}&a_{8}&a_{9}&a_{3}a_{6}-a_{4}a_{5}
\end{array}\right),$$
we have the restriction
$$\mu(e_2,e_3)=\frac{a_{3}\alpha_2 + a_{4}\alpha_3}{a_{3}a_{6}-a_{4}a_{5}}e_5, \quad \mu(e_2,e_4)=\frac{a_{5}\alpha_2 +a_{6}\alpha_3}{a_{3}a_{6}-a_{4}a_{5}}e_5,\quad \mu(e_3,e_4)=\frac{a_{4}(a_{3}\alpha_4+a_{5}\alpha_5)+a_{6}(a_{3}\alpha_5+a_{5}\alpha_6)}
{a_{3}a_{6}-a_{4}a_{5}}e_5,$$
$$\mu(e_2,e_2)=\frac{\alpha_1}{a_{3}a_{6}-a_{4}a_{5}}e_5,\quad \mu(e_3,e_3)=\frac{a_{3}^2\alpha_4+2a_3a_{5}\alpha_5+a_{5}^2\alpha_6}
{a_{3}a_{6}-a_{4}a_{5}}e_5,\quad \mu(e_4,e_4)=\frac{a_{4}^{2}\alpha_4+2a_{4}a_{6}\alpha_5+a_{6}^{2}\alpha_6}{a_{3}a_{6}-a_{4}a_{5}}e_5.$$

\begin{itemize}
\item Let $\alpha_5^2-\alpha_4\alpha_6 \neq0,$ then choosing $a_3=- \frac{a_5\alpha_5+a_5\sqrt{\alpha_5^2-\alpha_4\alpha_6}}{\alpha_4},$ $a_4=- \frac{a_6\alpha_5-a_6\sqrt{\alpha_5^2-\alpha_4\alpha_6}}{\alpha_4},$ we obtain
$\mu(e_3,e_3)=\mu(e_4,e_4)=0.$ Thus, in this case, we may suppose $\alpha_4=\alpha_6=0$ and $\alpha_5\neq 0.$ Then we have $a_{4}=a_{5}=0$ (or $a_{3}=a_{6}=0$) and obtain that $$\mu(e_2,e_2)=\frac{\alpha_1}{a_{3}a_{6}}e_5,\quad \mu(e_2, e_3)=\frac{\alpha_2}{a_{6}}e_5,\quad \mu(e_2,e_4)=\frac{\alpha_3}{a_{3}}e_5, \quad \mu(e_3,e_4)=\alpha_5e_5$$
(or $\mu(e_2,e_2)=-\frac{\alpha_1}{a_{4}a_{5}}e_5,\quad \mu(e_2, e_3)=-\frac{\alpha_3}{a_{5}}e_5,\quad \mu(e_2,e_4)=-\frac{\alpha_2}{a_{4}}e_5,\quad \mu(e_3, e_4)=-\alpha_5e_5$).

It follows that in the case of $(\alpha_2, \alpha_3)\neq (0, 0)$ we may suppose $\alpha_3\neq0.$
Thus, we consider the following subcases:
\begin{itemize}
\item If $\alpha_3=0, \alpha_2=0, \alpha_1 \neq0,$ then choosing $a_{3}=\frac{\alpha_1}{a_{6}},$ we obtain the algebra $\mathcal{L}_{71}^{\alpha \neq 0}.$

\item If  $\alpha_3\neq 0, \alpha_2=0,\alpha_1=0,$ then choosing $a_{3}=\alpha_3,$ we obtain the algebra $\mathcal{L}_{72}^{\alpha \neq 0}.$

\item If  $\alpha_3\neq 0, \alpha_2=0,\alpha_1\neq0,$ then choosing $a_{3}=\alpha_3,$ $a_6 = \frac{\alpha_1} {\alpha_3},$ we obtain the algebra $\mathcal{L}_{73}^{\alpha \neq 0}.$

	\item If $\alpha_3 \neq0, \alpha_2 \neq0,$ then choosing $a_{3}=\alpha_3, a_{6}=\alpha_2,$ we have the algebra $\mathcal{L}_{74}^{\alpha \neq 0, \beta}.$

\end{itemize}

\item Let $\alpha_5^2-\alpha_4\alpha_6=0,$ then choosing $a_3 = -\frac {a_5\alpha_5} {\alpha_4}, $ we can suppose $\alpha_4=\alpha_5=0.$
%

    If $\alpha_6=0,$ then we have
    $$\mu(e_2,e_2)=\frac{\alpha_1}{a_{3}a_{6}-a_{4}a_{5}}e_5,\quad \mu(e_2,e_3)=\frac{a_{3}\alpha_2 + a_{4}\alpha_3}{a_{3}a_{6}-a_{4}a_{5}}e_5, \quad \mu(e_2,e_4)=\frac{a_{5}\alpha_2 +a_{6}\alpha_3}{a_{3}a_{6}-a_{4}a_{5}}e_5,$$
and consider the following cases:

\begin{itemize}
\item If $\alpha_3=0, \alpha_2=0, \alpha_1 \neq0,$ then choosing $a_{3}=\frac{\alpha_1}{a_{6}},$ $a_5=0,$ we obtain the algebra $\mathcal{L}_{71}^{\alpha = 0}.$

\item If  $(\alpha_3, \alpha_2)\neq (0, 0),$ $\alpha_1=0,$ then we may suppose $\alpha_3 \neq 0$ and choosing $a_{3}=\alpha_3,$ $a_{4}=-\alpha_2,$ $a_{5}=0,$ we obtain the algebra $\mathcal{L}_{72}^{\alpha = 0}.$

\item If  $(\alpha_3, \alpha_2)\neq (0, 0),$ $\alpha_1\neq0,$ then we may suppose $\alpha_3 \neq 0$ and choosing  $a_{3}=\alpha_3,$ $a_{4}=-\alpha_2,$ $a_{5}=0,$ $a_6 = \frac{\alpha_1}{\alpha_3},$ we obtain the algebra $\mathcal{L}_{73}^{\alpha = 0}.$

\end{itemize}

If $\alpha_6\neq0,$ then $a_{5}=0$ and we have
$$\mu(e_2,e_2)=\frac{\alpha_1}{a_{3}a_{6}}e_5,\quad \mu(e_2,e_3)=\frac{a_{3}\alpha_2 + a_{4}\alpha_3}{a_{3}a_{6}}e_5,\quad \mu(e_2,e_4)=\frac{\alpha_3}{a_{3}}e_5,\quad \mu(e_4,e_4)=\frac{a_{6}\alpha_6}{a_{3}}e_5.$$
Then we consider the following subcases:
\begin{itemize}
		
\item If $\alpha_3=0, \alpha_2=0,\alpha_1\neq0,$ then choosing $a_{3}=a_{6}\alpha_6,$ $a_6 = \sqrt{\frac{\alpha_1}{\alpha_6}},$ we have the algebra $\mathcal{L}_{75}.$

\item If $\alpha_3=0, \alpha_2\neq0,$ then choosing $a_{3}=\alpha_2\alpha_6,$ $a_6 = \alpha_2,$ we have the algebra $\mathcal{L}_{76}^{\alpha}.$

    \item If $\alpha_3\neq0,$ then choosing $a_{3}=\alpha_3,$ $a_4 = -\alpha_2,$
     $a_6 = \frac{\alpha_3}{\alpha_6},$ we have the algebra $\mathcal{L}_{77}^{\alpha}.$

\end{itemize}
\end{itemize}
\end{proof}

\section*{Funding}

This work is supported by grant ''Automorphisms of operator algebras, classifications of infinite-dimensional non associative algebras and superalgebras``, No. FZ-202009269,
Ministry of higher education, science and innovations of the Republic of Uzbekistan, 2021-2025.

\section*{Data availability}

No data was used for the research described in the article.

\section*{Conflict of Interest.} The authors have no conflicts and interest to declare that are relevant to the content of this article.

\section*{Acknowledgments.} The authors are grateful to the referee of the paper for his(her) useful comments.

\end{document}